\documentclass[graybox]{svmult}

\usepackage{mathptmx}       
\usepackage{helvet}         
\usepackage{courier}        
\usepackage{type1cm}        
\usepackage{makeidx}         
\usepackage{graphicx}        
\usepackage{multicol}        
\usepackage[bottom]{footmisc}

\usepackage{amsfonts}

\usepackage{amssymb}
\usepackage{amsmath}
\usepackage{float}
\usepackage{xcolor}
\usepackage{epsfig}

\def\bS{{\bf S}}
\newtheorem{prop}{Proposition}

\makeindex             

\begin{document}

\title*{Binary Mean Field Stochastic Games: Stationary Equilibria   and Comparative Statics\thanks{in IMA volume {\it Modeling, Stochastic Control, Optimization, and Applications}, eds. G. Yin and Q. Zhang, Springer, 2019, p. 283-313. Submitted Dec 2018; revised Feb 2019. This version: Oct 10, 2020. Minor changes in Example~2}}
\author{Minyi Huang and Yan Ma}
\institute{Minyi Huang \at   School of Mathematics and Statistics, Carleton University, Ottawa, ON
K1S 5B6, Canada  \email{mhuang@math.carleton.ca.} This author was supported by the Natural Sciences and Engineering Research Council (NSERC) of Canada 
\and Yan Ma \at School of Mathematics and Statistics, Zhengzhou University, 450001,  Henan, China
 \email{mayan203@zzu.edu.cn.} This author was supported by the National Science Foundation of China (No. 11601489)  }
%
%
\maketitle

\abstract*{This paper considers  mean field games in a multi-agent Markov decision process (MDP) framework. Each player has a continuum  state and  binary action, and benefits from the improvement of the condition of the overall population. Based on an infinite horizon discounted individual cost, we show existence of a stationary equilibrium, and prove its uniqueness under a positive externality condition. We further analyze comparative statics of the stationary equilibrium by quantitatively determining the impact of the effort cost.}

\abstract{This paper considers  mean field games in a multi-agent Markov decision process (MDP) framework. Each player has a continuum  state and  binary action, and benefits from the improvement of the condition of the overall population. Based on an infinite horizon discounted individual cost, we show existence of a stationary equilibrium, and prove its uniqueness under a positive externality condition. We further analyze comparative statics of the stationary equilibrium by quantitatively determining the impact of the effort cost.}

\section{Introduction}

Mean field game theory provides a powerful methodology for reducing complexity in the analysis and design of strategies in large population dynamic games \cite{HCM03,HMC06,LL07}.
Following ideas in statistical physics, it takes a continuum approach to specify the aggregate impact of many individually insignificant players and solves a special stochastic optimal control problem from the point of view of a representative player.
  By this methodology, one may
construct a set of decentralized strategies
for the original large but finite population model and show its $\varepsilon$-Nash equilibrium property \cite{HCM03,HCM07,HMC06}.
 A related solution notion in Markov decision models is the oblivious equilibrium \cite{WBR08}.
  The readers are referred to \cite{BFY13,C14,CHM17,C12,CD18} for an overview on mean field game theory and further references. For mean field type optimal control, see
\cite{BFY13,Y11}, but the analysis in these models only involves a single decision maker.

Dynamic games within an MDP setting originated from the work of  Shapley  and are called stochastic games \cite{FV97,S53}. Their mean field game extension has been studied in the literature; see e.g. \cite{AJW15,B15,STR18,WBR08}.
Continuous time mean field games with finite state space can be found in \cite{GMS10,K12}.
Our previous work \cite{HM16Chen,HM17}
studied a class of mean field games in a multi-agent Markov decision process (MDP) framework.  The players in \cite{HM16Chen}  have continuum state spaces and binary action spaces, and have  coupling through their costs. The state of each player is used to model its risk (or unfitness) level, which has random increase if no active control is taken.  Naturally, the one-stage cost of a player is an increasing function of its own state apart from coupling with others. The motivation of this modeling framework comes from  applications including network security investment games and flue vaccination games \cite{JAW11,LB08,MP10};
when the one-stage cost is  an increasing function of the population average state, it reflects positive externalities. Markov decision processes with binary action spaces also arise in control of queues
and machine replacement problems \cite{AS95,BR11}. Binary choice models have formed a subject of
significant interest \cite{B13,BD01,S73,S95,WWA11}.
Our game model has connection with anonymous sequential games
\cite{JR88}, which combine stochastic game modeling with a continuum of players.
In anonymous sequential games one determines the equilibrium as a joint state-action distribution of the  population and  leaves the individual
strategies unspecified \cite[Sec. 4]{JR88}, although there is an interpretation of randomized actions for players sharing a given state.

For both anonymous games and MDP based mean field games, stationary solutions with discount have been studied in the literature \cite{AJW15,JR88}.  These works give more focus on fixed point analysis to prove the existence of a stationary distribution. This approach does not address ergodic behavior of individuals or the population while assuming the  population starts from the steady-state distribution at the initial time. Thus, there is a need to examine whether the individuals collectively have the ability to move into that distribution at all when they have a general initial distribution.
    Our  ergodic analysis based approach will provide justification of the stationary solution regarding the population's ability to settle down around the limiting distribution.

The previous work \cite{HM16Chen, HM17} studied the finite horizon mean field game by showing existence of a solution with threshold policies, and  under an infinite horizon discounted cost further proved there is at most one stationary equilibrium for which existence was not established.
A similar continuous time modeling is introduced in \cite{ZH17}, which addresses Poisson state jumps and impulse control. It should be noted that except for linear-quadratic models \cite{B12,HCM07,HZ18,LZ08,MB17}, mean field games rarely have
closed-form solutions and often rely on heavy numerical computations. Within this context, the consideration of structured solutions, such as threshold policies, is of particular interest  from the point of view of efficient computation and simple implementation. Under such a policy, the individual states evolve as regenerative processes \cite{A03,SW93}.

By exploiting stochastic monotonicity, this paper adopts more  general state transition assumptions than in \cite{HM16Chen, HM17} and continues the analysis on
 the stationary equation
 system. The first contribution of the present paper is the proof of the existence of a stationary equilibrium. Our analysis depends on checking the continuous dependence of the limiting state distribution on the threshold parameter in the  best response.
The existence and uniqueness analysis in this paper has appeared in a preliminary form in the conference paper \cite{HMcdc17}.

A key parameter in our game model is the effort cost. Intuitively, this parameter is a disincentive indicator of an individual for taking active efforts, and in turn will further impact the mean field forming the ambient environment of that agent. This suggests that we can study a family of mean field games parametrized by the effort costs and compare their solution behaviors.
We  address this in the setup of comparative statics, which have a long history in the economic literature   \cite{H39,MS94,S83} and operations research \cite{T98} and provide the primary means to analyze the effect of model parameter variations. For dynamic models, such as economic growth models, the analysis follows similar ideas and  is sometimes  called comparative dynamics \cite{A96,B85,O73,S83}  by comparing two dynamic equilibria. In control and optimization, such studies are usually called sensitivity analysis \cite{BS00,D63,IK92}. For comparative statics in large static  games and  mean field games, see \cite{AJ13,AJ15}.
Our analysis is accomplished by performing perturbation analysis around the equilibrium of the mean field game.

The paper is organized as follows. Section \ref{sec:mean} introduces the mean field stochastic game. The best response is analyzed in Section \ref{sec:br}.
 Section \ref{sec:se} proves existence and uniqueness of stationary  equilibria.
Comparative statics are analyzed in Section \ref{sec:cs}.
 Section \ref{sec:con} concludes the paper.

\section{The  Markov Decision Process Model}
\label{sec:mean}

\subsection{Dynamics and Costs}

The system consists of $N$  players
denoted by ${\cal A}_i$, $1\le i\le N$.
At time $t\in \mathbb{Z}_+=\{0, 1,2, \ldots\}$, the state of ${\cal A}_i$ is denoted by $x_t^i$, and its action by $a_t^i$.
For simplicity, we consider a population of homogeneous (or symmetric) players.
Each player has  state space ${\bf S}=[0,1]$ and action space  ${\bf A}=\{a_0, a_1\}$.  A value of ${\bf S}$ may be interpreted as a risk or unfitness level.
A player can either
 take inaction (as $a_0$) or
 make an active effort (as $a_1$).
  For an interval $I$, let ${\cal B}(I)$ denote the Borel $\sigma$-algebra of $I$.

The state of each player evolves as a controlled Markov  process,
 which is affected only by its own action.
For $t\ge 0$ and $x\in \bS$,  the state has  a transition kernel specified  by
\begin{align}
&P(x_{t+1}^i\in B|x_t^i =x, a_t^i=a_0)= Q_0(B|x), \label{xa0}\\
&P(x_{t+1}^i=0|x_t^i =x, a_t^i=a_1)=1 , \label{xa1}
\end{align}
where $Q_0(\cdot|x)$ is a stochastic kernel defined for
$B\in {\cal B}(\bS) $
and $Q_0([x,1]|x)=1$.
 By the structure of $Q_0$, the state of the player deteriorates if no active control is taken.
The vector process $(x_t^1, \ldots x_t^N)$  constitutes a controlled Markov process in higher dimension with its transition kernel defining  a product measure on $({\cal B}(\bS))^N$ for given $(x_t^1,\cdots, x_t^N,a_t^1, \ldots, a_t^N)$.

Define the population average state
 $x^{(N)}_t= \frac{1}{N} \sum_{i=1}^N x_t^i$.
The one stage cost of ${\cal A}_i$   is
$$
c(x_t^i, x^{(N)}_t, a_t^i)= R(x_t^i, x^{(N)}_t) +\gamma 1_{\{a_t^i=a_1\}},
$$
where $\gamma>0$ and  $\gamma 1_{\{a_t^i=a_1\}}$ is the effort cost. The function $R\ge 0$  is defined on   $\bS\times\bS$ and models the risk-related cost. Let $\nu^i$ denote the strategy of ${\cal A}_i$.
 We  introduce
 the infinite horizon  discounted cost
\begin{align}
J_{i}(x_0^1, \ldots, x_0^N, \nu^1, \ldots, \nu^N)=
E \sum_{t=0}^\infty \beta^t c(x_t^i, x^{(N)}_t, a_t^i),
\quad 1\le i\le N. \label{jfb}
\end{align}

 The standard methodology of mean field games may be applied by approximating $\{x_t^{(N)}, t\ge 0\}$ by a deterministic sequence $\{z_t, t\ge 0\}$ which depends on the initial condition of the system.
One may solve the limiting optimal control problem of ${\cal A}_i$ and derive a dynamic programming equation for its value function denoted by $v_i(t,x, (z_k)_{k=0}^\infty)$, whose dependence on $t$ is due to the time-varying sequence $\{z_t, t\ge 0\}$. Subsequently one derives another equation for the mean field $\{z_t, t\ge 0\}$ by averaging the individual states across the population.
This approach,
 however, has the drawback of heavy computational load.

 \subsection{Stationary Equilibrium}

 We are interested in  a steady-state form of the
 solution of the mean field game starting with $\{z_t, t\ge 0\}$.
 Such steady state equations provide information on the long time behavior of the solution and  are of interest in their own right. They may also be used for approximation purposes to compute  strategies efficiently.
 We introduce the system
\begin{align}
&v(x) =  \min \Big[\beta \int_0^1 v( y) Q_0(dy|x) + R(x, z),  \quad \beta v(0) + R(x,z)+ \gamma\Big],  \label{dpvb0}\\
&z= \int_0^1x \mu(dx),\label{zxt0}\
\end{align}
where $\mu$ is a probability measure on ${\bf S}$.
We say $(v, z,  \mu,a^i(\cdot) )$ is  a {\it stationary equilibrium} to
\eqref{dpvb0}-\eqref{zxt0}
if i) the feedback policy $ a^i(\cdot)$, as a mapping from $\bS$ to $\{a_0, a_1\}$, is the best response with respect to $ z$
in \eqref{dpvb0}, ii) given an initial distribution of $x_0^i$,  $\{x_t^i, t\ge 0\}$ under the policy $ a^i$ has its distribution converging (under a total variation norm or only weakly) to the stationary distribution (or called limiting distribution) $ \mu$.

We may interpret $v$ as the value function of an MDP with cost
$\bar J_{i}(x_0^i,z, \nu^i)=   E \sum_{t=0}^\infty
\beta^t c(x_t^i, z, a_t^i)$.
An alternative way to interpret \eqref{dpvb0}-\eqref{zxt0} is that the initial state of ${\cal A}_i$ has been sampled according to the ``right" distribution $\mu$, and that $z$ is obtained by averaging an infinite number of such initial values by the law of large numbers \cite{S06}.
A similar solution notion is adopted in \cite{AJ15,AJW15} but
ergodicity is not part of their solution specification.

Let the probability measure  $\mu_k$ be the distribution  of
$\mathbb{R}$-valued random variable $Z_k$, $k=1, 2$.
We say $\mu_2$   stochastically dominates  $\mu_1$, and denote $\mu_1\le_{st} \mu_2$, if $\mu_2((y, \infty)) \ge \mu_1((y, \infty))$    (or equivalently,  $P(Z_2>y)\ge P(Z_1>y)$) for all $y$. It is well known \cite{MS02} that $\mu_1\le_{st} \mu_2$ if and only if
 \begin{align}
 \int \psi(y) \mu_1(dy)\le \int \psi(y) \mu_2(dy) \label{psi12}
 \end{align}
 for all increasing function $\psi$ (not necessarily strictly increasing) for which the two integrals are finite.
A stochastic kernel ${\cal Q}(B|x)$, $0\le x\le 1$, $B\in {\cal B}({\bf S}) $, is said to be strictly stochastically increasing if $\varphi(x):= \int_{\bf S}\psi (y){\cal Q}(dy|x)$ is strictly increasing in $x\in {\bf S}$ for any strictly increasing function $\psi: [0, 1]\to \mathbb{R}$ for which the integral is necessarily finite.
${\cal Q}(\cdot|x)$ is said to be weakly continuous if $\varphi$ is continuous whenever $\psi$ is continuous.

Let $\{Y_t, t\ge 0\}$ be a Markov process with state space $[0,1]$,
transition kernel $Q_0(\cdot|x)$ and initial state $Y_0=0$. So each of its trajectories is monotonically increasing. Define $\tau_{Q_0}^\theta=\inf\{t|Y_t \ge \theta\}$ for $\theta\in (0, 1)$. It is clear that  $\tau_{Q_0}^{\theta_1} \le  \tau_{Q_0}^{\theta_2}$ for $0<\theta_1<\theta_2<1$.

The following assumptions are introduced.

\begin{itemize}

\item[(A1)] \qquad $\{x_0^i, i\ge 1\}$ are i.i.d. random variables taking values in $\bS$.

\item[(A2)]\qquad $R(x,z)$ is a continuous function  on  $\bS\times \bS$.
  For each fixed $z$, $R(\cdot, z)$ is  strictly increasing.

\item[(A3)] \qquad i) $Q_0(\cdot |x)$ satisfies $Q_0([x,1]|x)=1$ for any $x$, and is
strictly  stochastically increasing; ii)
 $Q_0(dy|x)$ is weakly continuous and has a positive probability density $q(y|x)$ for each fixed $x<1$; iii) for any small $0<\delta<1$,  $\inf_xQ_0([1-\delta, 1 ]|x)>0$.

\item[(A4)]\qquad $R(x, \cdot)$ is increasing for each fixed $x$.

\item[(A5)]\qquad  $\lim_{\theta\uparrow 1}E\tau_{Q_0}^\theta =\infty$.
\end{itemize}

(A3)-iii) will be used to ensure the uniform ergodicity of the controlled Markov process.
In fact, under (A3) we can show $E\tau_{Q_0}^\theta <\infty$.
The following condition is a special case of (A3).

\begin{itemize}
\item[] (A3$^\prime$) There exists a random variable such that
$Q_0(\cdot|x)$ is equal to the law of $x+(x-1)\xi$ for some random variable $\xi$ with probability  density $f_\xi(x)>0$, a.e. $x\in \bS$.

\end{itemize}

When (A3$^\prime$) holds, we can verify (A5) by analyzing the stopping time $\tau_\xi=\inf\{t|\prod_{s=1}^t \xi_s\le 1-\theta\}$, where $\{\xi_s, s\ge 1\}$ is a sequence of i.i.d. random variables with probability density $f_\xi$.  For existence analysis of the mean field game, (A5) will be used to ensure continuity of the mean field when the threshold $\theta$ approaches $1$.

\begin{prop}
The two conditions are equivalent:

i) $\mu_1\le_{st} \mu_2$, and $\mu_1\ne \mu_2$;

ii)  $\int_{\mathbb R} \phi(y)\mu_1(dy) <\int_{\mathbb R} \phi(y) \mu_2(dy)$ for all     strictly increasing function $\phi$ for which both integrals are finite.
\end{prop}

\proof Assume i) holds.  By \cite[Theorem 1.2.16]{MS02}, we have
\begin{align}
\phi(Z_1)\le_{st} \phi(Z_2), \label{h1h2}
\end{align}
and so $E\phi(Z_1)\le E\phi(Z_2)$.  Since $\mu_1\ne \mu_2$, there exists $y_0$ such that $P(Z_1> y_0)\ne P(Z_2>y_0)$. Take $r$ such that $\phi(y_0)=r$. Then \begin{align}
P(\phi(Z_1)>r)\ne P(\phi(Z_2)>r). \label{pz12}
\end{align}

 If  $E\phi(Z_1) =E\phi(Z_2)$ were true, by \eqref{h1h2}  and \cite[Theorem 1.2.9]{MS02},  $\phi(Z_1)$ and $\phi(Z_2)$  would have the same distribution, which contradicts
 \eqref{pz12}. We conclude $E\phi(Z_1)<E\phi(Z_2)$, which is equivalent to ii).

Next we show ii) implies i).  Let $\psi$ be any increasing function satisfying \eqref{psi12} with two finite integrals. When ii) holds, we take $\phi_\epsilon=  \psi +\frac{\epsilon y }{1+|y|}$, $\epsilon>0$. Then $\int \phi_\epsilon \mu_1(dy) < \int \phi_\epsilon \mu_2(dy)$ holds for all $\epsilon >0$.
  Letting $\epsilon \to 0$, then \eqref{psi12} follows and $\mu_1\le_{st} \mu_2$. It is clear $\mu_1\ne \mu_2$. \qed

\section{Best Response}
\label{sec:br}

For this section we assume (A1)-(A3).
 We take  any fixed $z\in [0,1]$ and consider \eqref{dpvb0} as a separate equation, which is rewritten below:
\begin{align}
v(x)= \min \Big\{
\beta\int_0^1 v (y)Q_0(dy|x) +R(x,z), \quad \beta v(0) +R(x, z)+\gamma
\Big\}. \label{dpb}
\end{align}
Here $z$  is not required to satisfy \eqref{zxt0}.
In relation to the mean field game, the resulting optimal policy will be called the
best response with respect to $z$.
Denote $G(x)= \int_0^1 v( y) Q_0(dy|x)$.

\begin{lemma} \label{lemma:LV}
i) Equation \eqref{dpb} has a unique solution $v\in C([0,1], \mathbb{R})$.

ii) $v$ is strictly increasing.

iii) The optimal policy is determined as follows:

\quad a) If $\beta G(1)< \beta v(0)+\gamma$, $a^i(x)\equiv a_0$.

\quad b) If $\beta G(1)= \beta v(0)+\gamma$, $a^i(1)= a_1$ and $a^i(x)= a_0$ for $x<1$.

\quad c) If $\beta G(0)\ge \beta v(0)+\gamma$,    $a^i(x) \equiv a_1$.

\quad d) If $\beta G(0)< \beta v(0)+\gamma <\rho G(1)$, there exists a unique $x^*\in (0, 1)$ and $a^i$ is a threshold policy with parameter $x^*$, i.e., $a^i(x)=a_1$ if $x\ge x^*$ and $a^i(x)=a_0$ if $x<x^*$.
\end{lemma}

\begin{proof}
Define the dynamic programming operator
\begin{align}
({\cal L} g)(x)=\min \Big\{
\beta\int_0^1 g (y)Q_0(dy|x) +R(x,z), \quad \beta g(0) +R(x, z)+\gamma
\Big\},\label{dpo}
\end{align}
which is from $ C([0,1], \mathbb{R})$ to itself.
 The proving method in \cite{HM16Chen}, \cite[Lemma 6]{HM17},   which assumed (A3$^\prime$), can be
extended to the present equation \eqref{dpb}
in a straightforward manner.

 In particular, for the proof of ii) and iii), we  obtain progressively  stronger properties of $v$ and $G$.
First, denoting $g_0=0$ and $g_{k+1}={\cal L} g_k$ for $k\ge 0$,
  we  use a successive approximation procedure to show that
$v$ is increasing, which implies that $G$ is continuous and increasing by weak continuity and  monotonicity of $Q_0$. Since $R$ is strictly increasing in $x$, by the right hand side of \eqref{dpb}, we show that $v$ is strictly increasing, which implies the same property for $G$ by strict monotonicity of $Q_0$.~\qed
\end{proof}

For the optimal policy specified in part iii) of Lemma \ref{lemma:LV}, we can
formally denote the threshold parameters for the corresponding cases: a) $\theta=1^+$, b) $\theta=1$, c) $\theta=0$, and d) $\theta=x^*$.
Such a policy will be called a $\theta$-threshold policy.
We give the condition for $\theta=0$ in the best response.

\begin{lemma}
For $\gamma>0$ and $v$ solving \eqref{dpb},
  \begin{align}
  \beta G(0)\ge \beta v(0)+\gamma \label{bgv0}
  \end{align}
  holds  if and only if
\begin{align}
\gamma\le \beta \int_0^1 R(y,z)Q_0(dy|0) - \beta R(0, z).\label{gaub}
\end{align}
\end{lemma}

\proof We show necessity first.
Suppose \eqref{bgv0} holds. Note that $G(x)$ is strictly increasing on $[0, 1]$.  Equation \eqref{dpb} reduces to
\begin{align}
&v(x)= \beta v(0)+R(x, z)+\gamma, \label{vg} \\
& \beta G(x)\ge \beta v(0) +\gamma,\quad  \forall x.  \label{vgge}
\end{align}
From \eqref{vg}, we uniquely solve
\begin{align}
v(0)= \frac{1}{1-\beta} [R(0,z)+\gamma], \quad v(x)=
\frac{\beta}{1-\beta} [R(0,z)+\gamma] +R(x, z)+\gamma, \label{vso}
\end{align}
which combined with \eqref{vgge} implies \eqref{gaub}.

We continue to show sufficiency. If $\gamma>0$ satisfies \eqref{gaub}, we use \eqref{vso} to construct $v$ and verify \eqref{vg} and \eqref{vgge}. So $v$ is the unique solution of \eqref{dpb} satisfying   \eqref{bgv0}.   \qed

The next lemma gives the condition for $\theta=1^+$
in the best response.

\begin{lemma}
For $\gamma >0$ and $v$ solving \eqref{dpb}, we have
\begin{align}
\beta G(1)< \beta v(0)+\gamma  \label{bg1v}
\end{align}
if and only if
\begin{align}
\gamma > \beta [V_\beta(1)-V_\beta(0)],  \label{glb}
\end{align}
where $V_\beta(x)\in C([0,1], \mathbb{R})$ is the unique solution of
\begin{align}  V_\beta(x)= \beta\int_0^1 V_\beta (y)Q_0(dy|x) +R(x,z). \label{ven}
\end{align}
\end{lemma}

\proof
By Banach's fixed point theorem, we can show that \eqref{ven} has a unique solution. Next, by a successive approximation $ \{V_\beta^{(k)}, k\ge 0 \}$ with $V_\beta^{(0)}=0$ in the fixed point equation, we can further show
that $V_\beta$ is strictly increasing. Moreover, $\int_0^1 V_\beta(y) Q_0(dy|x) $ is increasing in $x$ by monotonicity of $Q_0$.

We show necessity. Since $G$ is strictly increasing, \eqref{bg1v} implies that the right hand side of \eqref{dpb} now reduces to the first term within the parentheses and that $v=V_\beta$. So \eqref{glb} follows.

To show sufficiency, suppose \eqref{glb}  holds.  We have
$$
\beta\int_0^1 V_\beta(y) Q_0(dy|x)\le \beta V_\beta(1) <\beta V_\beta(0)+\gamma, \quad \forall x.
$$
Therefore, $v:= V_\beta$ gives the unique solution of  \eqref{dpb} and $\beta G(1)< \beta v(0)+\gamma$. \qed

\begin{example}
Let $R(x, z)= x(c+z)$, where $c>0$.  Take $Q_0(\cdot|x)$ as uniform distribution on $[x, 1]$.
Then \eqref{ven} reduces to
$$
V_\beta(x)= \frac{\beta}{1-x} \int_x^1V_\beta(y)dy +R(x, z).
$$
Define $\phi(x)= \int_x^1 V_\beta(y)dy$, $x\in [0, 1]$.
Then $\phi'(x)= -\frac{\beta}{1-x}\phi(x) -R(x, z)$ holds and we solve
$$
\phi(x)= (1-x)^\beta \int_x^1 \frac{R(s, z)}{(1-s)^\beta} ds,
$$
where the right hand side converges to 0 as $x\to 1^-$.
We further obtain
$$
V_\beta(x)= \beta (1-x)^{\beta-1}\int_x^1 \frac{R(s, z)}{(1-s)^\beta} ds +R(x, z)
$$
for $x\in [0, 1)$, and the right hand side has the limit $\frac{R(1,z)}{1-\beta}$
as $x\to 1^-$. This gives a well defined $V_\beta \in C([0, 1], \mathbb{R})$.
Therefore,
$
V_\beta(0)= \frac{\beta (c+z)}{(1-\beta)(2-\beta)}.
$
Then \eqref{glb} reduces to
$
\gamma > \frac{2\beta(c+z)}{ 2-\beta  }.
$
\end{example}

\section{Existence of  Stationary Equilibria}
\label{sec:se}

Assume (A1)-(A5) for this section.
Define the class ${\cal P}_0$ of probability measures  on $\bS$ as follows:
 $\nu\in {\cal P}_0$ if
there exist a constant $c_\nu\ge  0$ and a Borel measurable function $g(x)\ge 0$ defined on $[0,1]$ such that
$$
\nu(B) = \int_B g(x) dx +c_\nu 1_B(0),
$$
where $B\in {\cal B}(\bS)$ and $1_B$ is the indicator function of $B$.
When restricted to $(0, 1]$, $\nu$ is absolutely continuous with respect to the Lebesgue measure $\mu^{\rm Leb}$.

Let $X$ be a random variable with distribution $\nu\in {\cal P}_0$.
Set $x_t^i=X
$. Define $Y_0=x_{t+1}^i$ by applying $a_t^i\equiv a_0$.
Further define $Y_1=x_{t+1}^i$ by applying the $r$-threshold policy $a_t^i$ with  $r\in (0, 1)$.

\begin{lemma}  \label{lemma:pi0}
The distribution $\nu_i$ of  $Y_i$ is in ${\cal P}_0$ for $i=0,1$.
\end{lemma}

\begin{proof}
Let $q(y|x)$ denote the density function of $Q_0(\cdot|x)$ for $x\in [0, 1)$, where $q(y|x)=0$ for $y<x$.  Denote
$$
g_0(y)= \int_{0\le x<y} q(y|x) \nu(dx), \quad y\in (0, 1),
$$
and
$$
g_1(y)= \int_{0\le x< y\wedge r} q(y|x) \nu (dx), \quad y\in (0, 1).
$$
Then it can be checked that
$$
P(Y_0\in B)=\int_B g_0(y) dy , \quad P(Y_1\in B)= \int_B g_1(y) dy +P(X\ge r) 1_B(0).
$$
This completes the lemma. \qed
\end{proof}

In order to show that \eqref{dpvb0}-\eqref{zxt0} has a solution, we define a mapping $\Gamma$: ${\bf S}\to {\bf S}$ by the following rule.
For $z\in [0,1]$, we solve \eqref{dpvb0} to obtain a well defined threshold
$\theta(z)\in [0, 1]\cup \{1^+\}$, which in turn determines a limiting distribution $\mu_{\theta(z)}$  of the closed-loop state process $x_t^i$
  by Lemma \ref{lemma:limd}. Define
 $$
 \Gamma (z)= \int_0^1 x \mu_{\theta(z)}(dx).
 $$
 If $\Gamma$ has a fixed point, we obtain a solution to \eqref{dpvb0}-\eqref{zxt0}.

 We analyze the case where the best response gives a strictly positive threshold.
Assume
\begin{align}
\gamma>\beta\max_{z\in [0,1]} \int_0^1 [R(y,z)-R(0,z)] Q_0(dy|0). \label{ggb}
\end{align}
Note that under  a zero threshold policy,
the behavior of the state process is sensitive to a positive perturbation of the threshold.
 The above condition ensures that  the zero threshold will not occur, and this will  ensure continuity of $\Gamma$ to  facilitate the fixed point analysis.

\begin{lemma} \label{lemma:Gam} Assume \eqref{ggb}.
Then $\Gamma(z)$ is continuous on $[0,1]$.
\end{lemma}

\proof Let $z_0\in [0, 1]$ be fixed, giving a corresponding threshold parameter   $\theta_0$ when \eqref{dpb} is solved using $z_0$.    We check continuity at $z_0$ and
 consider 3 cases.

Case i) $\theta_0\in (0,1)$.  Let $\pi_0$ be the
stationary distribution with the $\theta_0$-threshold  policy.  Consider any fixed $\epsilon>0$.
There exists $\epsilon_1$ such that for all $\theta\in (\theta_0-\epsilon_1, \theta_0+\epsilon_1)\subset (0, 1)$, $|\int_0^1 x\pi(dx)-\int_0^1 x\pi_0(dx)|
<\epsilon$, where $\pi$ is the stationary distribution associated with $\theta$. This follows since $\lim_{\theta\to \theta_0}\|\pi-\pi_0\|_{\rm TV}=0$ by Lemma \ref{lemma:sipi}.
 Now by the continuous dependence of the solution of the dynamic programming equation on $z$, we can select a sufficiently small $\delta>0$ such that for all $|z-z_0|<\delta$, $z$ generates a threshold parameter $\theta\in (\theta_0-\epsilon_1, \theta_0+\epsilon_1)$, which implies $
 |\Gamma (z)-\Gamma(z_0)|\le \epsilon$.

Case ii) $z_0$ gives $\theta_0=1$. Then $\Gamma (z_0)=1$.
Fix any $\epsilon>0$. Then we can show there exists $\epsilon_1$ such that  for all $\theta\in (1-\epsilon_1, 1)$, the associated stationary distribution $\pi_\theta$ gives $ |\Gamma(z_0)- \int_0^1 x \pi_\theta(dx)|<\epsilon$, where we use (A5) and the right hand  side of \eqref{lrun} to estimate a lower bound for $ \int_0^1 x \pi_\theta(dx)$.
Now, there exists $\delta>0$ such that any $z$ satisfying $|z-z_0|<\delta$
gives a threshold $\theta$ either in $(1-\epsilon_1, 1)$ or equal to 1 or $1^+$; for each case, we have
$ |\Gamma(z_0)- \int_0^1 x \pi_\theta(dx)|<\epsilon$.

Case iii) $z_0$ gives $\theta_0=1^+$. Then there exists $\delta>0$ such that any
$z$ satisfying $|z-z_0|<\delta$ gives a threshold parameter
$\theta=1^+$. Then $\Gamma (z)=\Gamma (z_0)=1$. \qed

\begin{theorem}\label{theorem:mfgdi}
Assume \eqref{ggb}. There exists a stationary equilibrium to \eqref{dpvb0}-\eqref{zxt0}.
\end{theorem}

\begin{proof}
Since $\Gamma$ is a continuous function from $[0, 1]$ to $[0,1]$ by Lemma \ref{lemma:Gam},   the theorem follows from Brouwer's fixed point theorem.    \qed
\end{proof}

Let  $x_t^{i,\theta}  $   and $\pi_\theta$ denote the state process and its  stationary distribution, respectively,  under a $\theta$-threshold policy.
Denote
$z(\theta) = \int_0^1 x \pi_\theta(dx).$
We have the first comparison  theorem on  monotonicity.
\begin{lemma}\label{theorem:zmono}
 $z(\theta_1)\le z(\theta_2)$ for $0<\theta_1<\theta_2<1$.
\end{lemma}

\proof
By the ergodicity of $\{x_t^{i,\theta_l}, t\ge 0\}$ in Lemma \ref{lemma:erg}, we have the representation
$
z(\theta_l)= \lim_{k\to \infty }  \frac{1}{k}
\sum_{t=0}^{k-1}x_t^{i, \theta_l}
$ w.p.1. Lemma
\ref{lemma:bb} implies
$z(\theta_1)\le z(\theta_2).$ \qed

To establish uniqueness, we consider
$ R(x, z)=R_1(x)R_2(z)$, where $R_1\ge 0$ and $R_2\ge 0$, and which satisfies (A1)-(A5). We further make the following assumption.

\begin{itemize}

\item[(A6)]\qquad  $R_2>0$ is strictly increasing on $\bS$.
\end{itemize}

This assumption indicates positive externalities since an individual  benefits from  the decrease of the population average state. This condition has a crucial role in the uniqueness analysis.

Given the product form of $R$, now \eqref{dpb} takes the form:
\begin{align*}
V(x) =  &\min \Big[\beta \int_0^1 V( y) Q_0(dy|x) + R_1(x)R_2(  z),
\quad \beta V(0) + R_1(x)R_2( z)+ \gamma\Big].
\end{align*}
Consider $0\le z_2<z_1\le 1$ and
\begin{align}
V_l(x) =  &\min \Big[\beta \int_0^1 V_l( y) Q_0(dy|x) + R_1(x)R_2(  z_l),  \quad \beta V_l(0) + R_1(x)R_2( z_l)+ \gamma\Big]. \label{Vlr12}
\end{align}
Denote the optimal policy as a threshold policy  with parameter $\theta_l$ in $[0, 1]$ or equal to $1^+$, where we follow the interpretation in Section \ref{sec:br}
if $\theta_l=1^+$.
We state the second comparison theorem about  the threshold parameters under different mean field parameters $z_l$.

\begin{theorem} \label{theorem:s1s2}
 $\theta_1$ and $\theta_2$ in \eqref{Vlr12} are
specified according to the following scenarios:

i) If $\theta_1=0$, then we have either $\theta_2\in [0, 1]$ or
$\theta_2=1^+$.

ii)   If $\theta_1\in (0,1)$, we have either  a)
$\theta_2\in (\theta_1, 1)$, or b) $\theta_2=1 $,
or c) $\theta_2= 1^+$.

iii) If $\theta_1=1$, $\theta_2=1^+$.

iv) If $\theta_1=1^+$, $\theta_2= 1^+$.
\end{theorem}

\proof Since $R_2(z_1)> R_2(z_2)>0$, we divide both sides of \eqref{Vlr12} by $R_2(z_l)$ and define
$
\gamma_l= \frac{\gamma}{R_2(z_l)}.
$
Then $0<\gamma_1<\gamma_2$. The dynamic programming equation reduces to \eqref{vlr12}.
Subsequently, the optimal policy is determined according to Lemma \ref{lemma:interv}. \qed

\begin{corollary}\label{cor:uni}
Assume (A6) in addition to the assumptions in Theorem \ref{theorem:mfgdi}. Then the system \eqref{dpvb0}-\eqref{zxt0}  has a unique stationary equilibrium.
\end{corollary}

\begin{proof}
The proof is similar to \cite{HM16Chen,HM17}, which assumed (A3$^\prime$). \qed.
\end{proof}

\section{Comparative Statics}
\label{sec:cs}

This section assumes (A1)-(A6).
Consider the two solution systems
\begin{align}\label{csvb}
\begin{cases}
\displaystyle
\bar v(x) =  \min \Big[\beta \int_0^1 \bar v( y) Q_0(dy|x) + R_1(x)R_2( \bar z),
\quad \beta \bar v(0) + R_1(x)R_2( \bar z)+ \bar \gamma\Big],\\
\displaystyle
\bar z=\int_0^1 x\bar  \mu(dx),
\end{cases}
\end{align}
and
\begin{align}\label{csv}
\begin{cases}
\displaystyle
v(x) =  \min \Big[\beta \int_0^1 v( y) Q_0(dy|x) + R_1(x)R_2(  z),
\quad \beta v(0) + R_1(x)R_2( z)+ \gamma\Big],\\
\displaystyle
z=\int_0^1 x \mu(dx).
\end{cases}
\end{align}
Suppose $\bar \gamma$ satisfies \eqref{ggb}.
By Corollary \ref{cor:uni}, \eqref{csvb} has a unique solution denoted by $(\bar v, \bar z,\bar \mu, \bar \theta)$, where $\bar\theta$ is  the threshold parameter. We further assume
$\bar \theta\in (0, 1)$.  Suppose $\gamma>\bar \gamma$.  Then we can uniquely solve $(v, z,\mu, \theta  )$. The next theorem presents a result on monotone comparative statics \cite{T98}.
\begin{theorem} \label{theorem:mcs}
If $\gamma >\bar \gamma$,
we have
$$
\theta >\bar \theta, \quad z >\bar z, \quad  v>\bar v.
$$
\end{theorem}
\proof We prove by contradiction. Assume $\theta\le \bar \theta$. Then by Lemma \ref{theorem:zmono}, $z\le \bar z$, and therefore, $ \frac{\gamma}{R_2(z)} >\frac{\bar \gamma}{R_2(\bar z)}$. By the method of proving  Theorem
\ref{theorem:s1s2}, we would establish $\theta>\bar \theta$, which contradicts the assumption $\theta\le \bar \theta$. We conclude $ \theta> \bar \theta$. By Lemma \ref{theorem:zmono} and Remark \ref{rem:b1}, we have $z> \bar z$.
For \eqref{csvb}, we use value iteration to approximate $\bar v$ by an increasing sequence of functions $\bar v_k$ with $\bar v_0=0$. Similarly, $v$ is approximated by $v_k$ with $v_0=0$. By induction, we have $ v_k\ge \bar v_k$ for all $k$. This proves $ v\ge \bar v$.

Next, we have $\beta  v(0) + R_1(x)R_2(  z)+  \gamma>
\beta \bar v(0) + R_1(x)R_2( \bar z)+ \bar \gamma$ on $[0,1]$, and
$\beta \int_0^1  v( y) Q_0(dy|x) + R_1(x)R_2(  z) >\beta \int_0^1 \bar v( y) Q_0(dy|x) + R_1(x)R_2( \bar z)$ on $(0, 1]$. By the method in \cite[Lemma 2]{HM16Chen}, we have
$v>\bar v$ on $(0, 1]$. Then $\int_0^1  v( y) Q_0(dy|0)> \int_0^1  \bar v( y) Q_0(dy|0)$. This further implies $v(0)>\bar v(0)$.
 \qed

\begin{remark}
It is possible to have  $\theta=1^+$ in Theorem \ref{theorem:mcs}.
\end{remark}

By a continuity argument, we can further show
$\lim_{\gamma\to \bar\gamma} (|\theta-\bar \theta| +|z-\bar z| +\sup_x|v(x)-\bar v(x)|)=0$.
In the analysis below, we take $\gamma=\bar \gamma +\epsilon$ for some small $\epsilon>0$. For this section, we  further introduce the following assumption.

\begin{itemize}
\item[(A7)] \qquad For $\gamma>\bar \gamma$,
$(v,z,\theta)$ has the representation
\begin{align}
&v(x)=\bar v(x)+ \epsilon w(x) +o(\epsilon), \qquad 0\le x\le 1, \\
& z = \bar z + \epsilon z_\gamma +o( \epsilon), \\
& \theta = \bar \theta + \epsilon \theta_\gamma +o( \epsilon),
\end{align}
where $v,z,\theta$ are solved  depending on the parameter $\gamma$ and $w$ is a function defined on $[0,1]$.
 The derivatives
$z_\gamma$ and $\theta_\gamma$ at $\bar \gamma$ exist, and  $R_2(z)$ is differentiable on $[0,1]$. For  $0\le x<1$,  the probability density function  $q( y|x)$, $y\in [x, 1]$, for $Q_0(dy|x)$  is continuous on $\{(x,y)|0\le   x\le y< 1\}$. Moreover,  $\frac{\partial q(y|x)}{\partial x}$ exists and is continuous in $(x,y)$.
\end{itemize}

We aim to provide a characterization of  $w, z_\gamma, \theta_\gamma$.

\begin{theorem}
The function $w$ satisfies
\begin{align} \label{weq}
w(x)=
\begin{cases}
 \displaystyle
 \beta \int_0^1 w(y)Q_0(dy|x)+R_1(x)R_2'(\bar z) z_\gamma,& 0\le x\le \bar \theta, \\
 \displaystyle
\beta w(0)+R_1(x)R_2'(\bar z)z_\gamma+1, &  \bar \theta<x\le 1.
\end{cases}
\end{align}
\end{theorem}

\proof
We have
$$
\bar v(x)= \beta \int_0^1 \bar v( y) Q_0(dy|x) + R_1(x)R_2(\bar z), \quad x\in [0, \bar \theta]
$$
and
$$
v(x)= \beta \int_0^1 v( y) Q_0(dy|x) + R_1(x)R_2(z), \quad x\in [0,  \theta].
$$
Note that $\theta> \bar \theta$. For any fixed $x\in[0, \bar \theta]$,
we have
$$
v(x)-\bar v(x)= \beta \int_0^1 (v(y)-\bar v(y)) Q_0(dy|x)+R_1(x)(R_2(z)-R_2(\bar z)).
$$
Then the equation of $w(x)$ for $x\in [0, \bar \theta]$ is derived.
We similarly treat the case $x\in (\bar \theta, 1]$. \qed

\begin{remark} \label{re:dis}
In general $w$ has discontinuity  at $x=\bar \theta$, so that
$\beta \int_0^1 w(y)Q_0(dy|\bar \theta)\ne  \beta w(0)+1.$
We give some interpretation. Let the value function be written as $v(x, \gamma)$ to explicitly indicate $\gamma$. Let the rectangle $[0,1]\times [\gamma_a, \gamma_b]$
 be a region of interest in which $(x,\gamma)$ varies so that the value function defines a continuous surface. Then $( \theta, \gamma)$ starts at $(\bar\theta, \bar\gamma)$ and traces out the curve of an increasing function along which the expression of the value function has a switch, and the value function surface may be visualized as two pieces glued together along the curve in a non-smooth way. The value of $w$ amounts to finding on the surface the directional derivative in the direction of $\gamma$; and therefore, discontinuity may occur at $x=\bar\theta$.
  \end{remark}

To better understand the solution of \eqref{weq},
we consider the  general equation
\begin{align}\label{wzc}
W(x)=
\begin{cases}
 \displaystyle
\beta \int_0^1 W(y)Q_0(dy|x)+R_1(x)R_2'(z_0) c_0, & 0\le x\le \theta_0,  \\
  \displaystyle
\beta W(0)+R_1(x)R_2'(z_0)c_0+1, &   \theta_0<x\le 1,
\end{cases}
\end{align}
where $c_0$, $z_0\in [0,1]$ and $\theta_0\in (0, 1)$ are arbitrarily chosen and fixed. Let $B([0,1], \mathbb{R})$ be the Banach space of bounded Borel measurable functions with norm $\|g\|=\sup_x |g(x)|$. By a contraction mapping, we can show \eqref{wzc} has a unique solution
$W\in B([0,1], \mathbb{R})$.

We continue to  characterize the  sensitivity $\theta_\gamma$ of the threshold.
Recall the partial derivative $\frac{\partial q(y|x)}{\partial x}$.
\begin{lemma}
We have
\begin{align} \label{bsiga}
\beta \Big[\int_{\bar \theta}^1 \bar v(y) \frac{\partial q(y|\bar\theta)}{\partial x} dy -\bar v(\bar \theta) q(\bar \theta|\bar \theta)\Big] \theta_\gamma=1+\beta w(0) -\beta \int_{\bar\theta}^1 w(y) Q_0(dy|\bar \theta).
\end{align}
\end{lemma}

\proof
Write $\gamma=\bar \gamma +\epsilon$. By the property of the threshold,   we have
\begin{align*}
\beta\int_{\bar \theta}^1 \bar v(y) Q_0(dy|\bar \theta)  =  \beta \bar v(0)+\bar \gamma,\quad
\beta \int_{\theta}^1 v(y) Q_0(dy|\theta)= \beta v(0)+\bar \gamma+\epsilon. \end{align*}
Note that $\theta>\bar \theta$. We check
\begin{align*}
\Delta:=&\int_{\theta}^1 v(y) Q_0(dy|\theta)-
\int_{\bar \theta}^1 \bar v(y) Q_0(dy|\bar \theta)\\
=& \int_{\theta}^1 v(y) Q_0(dy|\theta)-
\int_{\theta}^1 \bar v(y) Q_0(dy|\bar \theta)-
\int_{\bar \theta}^\theta \bar v(y) Q_0(dy|\bar \theta)\\
=&  \int_{\theta}^1 v(y) Q_0(dy|\theta)-
\int_{\theta}^1 \bar v(y) Q_0(dy| \theta)\\
& +  \int_{\theta}^1 \bar v(y) Q_0(dy| \theta)-
\int_{\theta}^1 \bar v(y) Q_0(dy|\bar \theta)-
\int_{\bar \theta}^\theta \bar v(y) Q_0(dy|\bar \theta)\\
=&\epsilon \int_\theta^1 w(y)q(y|\theta)dy +(\theta-\bar\theta) \int_\theta^1 \bar v(y) [\partial q(y|\theta) / \partial x]dy-(\theta-\bar \theta)\bar v(\bar\theta)q(\bar\theta|\bar\theta)\\
 &+o(\epsilon+|\theta-\bar \theta|)  \\
=&\epsilon \int_{\bar \theta}^1 w(y)q(y|\bar\theta)dy +(\theta-\bar\theta) \int_{\bar\theta}^1 \bar v(y) [\partial q(y|\bar\theta) / \partial x]dy-(\theta-\bar \theta)\bar v(\bar\theta)q(\bar\theta|\bar\theta)\\
&+o(\epsilon+|\theta-\bar \theta|).
\end{align*}
Note that
$$
\beta \Delta= \beta [v(0)-\bar v(0)] +\epsilon.
$$
We derive
$$
\beta \int_{\bar\theta}^1 w(y) Q_0(dy|\bar\theta)
+ \beta {\theta_\gamma} \int_{\bar\theta}^1 \bar v(y) \frac{\partial q(y|\bar \theta)}{\partial x} dy-\beta \bar v(\bar\theta) q(\bar\theta| \bar\theta) {\theta_\gamma} =\beta w(0)+1.
$$
This completes the proof.
\qed

\begin{lemma}
Given the threshold $\bar \theta\in (0, 1)$,
 the stationary distribution  $\bar \mu$ has a probability density function (p.d.f.) $p(x)$ on $(0, 1]$, and $\bar \mu(\{0\})= \pi_0$, where $(p, \pi_0)$ is determined by
\begin{align}
&\pi_0=\int_{\bar\theta}^1 p(x) dx, \label{mup0} \\
&{\displaystyle{ p(x)=
\begin{cases}
 \displaystyle
 \int_0^x q(x|y) p(y) dy +\pi_0q(x|0), &0\le  x<\bar \theta,\\
 \displaystyle
\int_0^{\bar \theta} q(x|y) p(y) dy+\pi_0 q(x|0), &  \bar \theta \le x\le 1.
\end{cases}}}\label{pinv}
\end{align}
\end{lemma}

\proof
Let $\delta_0$ be the dirac measure at $x=0$. For any Borel subset $B\subset [0,1]$, we have  $\bar \mu(B) =\int_0^1 [Q_0(B|y)1_{(y<\bar \theta)} +\delta_0(B) 1_{(y\ge \bar \theta)}  ] \bar \mu(dy)$. Then it can be checked that $(p,\pi_0 )$ satisfying the above equations determines the stationary distribution.
Now we show there exists a unique solution.
Let $\pi_0>0$ be a constant to be determined.
Consider the Volterra integral equation
\begin{align}
p(x)=  \int_0^x q(x|y) p(y) dy +\pi_0q(x|0), \quad 0\le  x\le \bar \theta, \end{align}
and we obtain a unique solution $p$ in  $C([0, \bar \theta], \mathbb{R})$ (see e.g. \cite[p.33]{K89}). In fact $p$ is a nonnegative function with
$\int_0^{\bar \theta}p(x) dx>0$. Subsequently, we further determine $p\ge 0$ on $[\bar \theta, 1]$ by \eqref{pinv}. The solution $p$ on $[0,1]$ depends linearly on $\pi_0$ and so there exists a unique $\pi_0$ such that $\int_0^1 p(x)dx +\pi_0=1$.   After we uniquely solve $p$ for \eqref{pinv}, we integrate both sides of this equation on $[0, 1]$ and obtain $\int_0^1 p(x) dx =\int_0^{\bar\theta} p(x) dx+\pi_0$, which implies that \eqref{mup0} is satisfied.
 \qed

\subsection{Special Case}

Now we suppose $Q_0(dy|x)$ has uniform distribution on $[x, 1]$ for all fixed $0\le x<1$, and
 $R(x,z)=R_1(x)R_2(z)= x(c+z)$, where $R_1(x)=x$, $R_2(z)=c+z$ and $c>0$.
 In this case, (A2)-(A6) are satisfied.
For \eqref{csvb}, we have
\begin{align}\label{vzs}
 \bar v(x)=
 \begin{cases}
\displaystyle
\frac{\beta}{1-x}\int_x^1 \bar v(y) dy +R_1(x)R_2(\bar z), & 0\le x\le \bar \theta,\\
\displaystyle
 \beta \bar v(0)+R_1(x)R_2(\bar  z) +\bar \gamma,&  \bar \theta\le  x\le  1 .
\end{cases}
\end{align}
Denote $\varphi (x)= \int_x^1 \bar v(y) dy $. Then
$$
\dot \varphi(x)= -\frac{\beta}{1-x} \varphi -R_1(x) R_2(\bar z), \quad 0\le x\le
\bar \theta.
$$
Taking the initial condition $\varphi (0)$, we have
$$
\varphi(x) =\varphi (0)(1-x)^\beta -(1-x)^\beta \int_0^x \frac{R_1(\tau) R_2 (\bar z)}{(1-\tau)^\beta} d\tau.
$$
On $[0, \bar \theta]$,
\begin{align*}
\bar v(x)&= (1-x)^{\beta-1} \bar v(0) -\beta (1-x)^{\beta-1} \int_0^x \frac{R_1(\tau)R_2( \bar z)}
{(1-\tau)^\beta} d\tau +R_1(x)R_2(\bar z)\\
&= (1-x)^{\beta-1}\Big[\bar v(0)-\frac{\beta(c+\bar z)}{(1-\beta)(2-\beta)}\Big]
+(c+\bar z)\Big[\frac{\beta}{(1-\beta)(2-\beta)}+\frac{2x}{2-\beta}\Big].
\end{align*}

By the continuity of $\bar v$ and its form on $[\bar \theta, 1]$, we have \begin{align}\label{vb0sg}
\bar v(\bar \theta)= \beta \bar v(0)+\bar\theta(\bar  z+c) +\bar \gamma. \end{align}
Hence,
\begin{align}
[(1-\bar \theta)^{\beta-1}-\beta ]\bar v(0)=\frac{ \beta (c+\bar z)
[(1-\bar \theta)^{\beta-1}-1]}{(1-\beta)(2-\beta)} -\frac{\beta  (c+\bar z)\bar\theta}{2-\beta}
+\bar\gamma.\label{univ0}
\end{align}

On the other hand, since $\bar v$ is increasing and $\bar \theta$ is the threshold, we have
\begin{align*}
\bar v(\bar\theta)& = \beta\int_{\bar \theta}^1 [\beta \bar v(0) +(c+z) y +\bar\gamma] \frac{1}{1-\bar \theta} dy  +(c+\bar z)\bar \theta\\
&= \beta^2 \bar v(0) +\beta \bar\gamma
+\frac{\beta(c+\bar z)}{2} +(\frac{\beta }{2} +1) (c+\bar z)\bar\theta,
\end{align*}
which combined with \eqref{vb0sg} gives
\begin{align}
\frac{\beta}{2} (c+\bar z)(1+\bar \theta)= (\beta
\bar v(0) +\bar\gamma) (1-\beta). \label{unith}
\end{align}

Given the special form of $Q_0(dy|x)$, \eqref{weq} becomes
\begin{align} \label{wequni}
w(x)=
\begin{cases}
  \displaystyle
\frac{\beta}{ 1-x} \int_x^1 w(y)dy+R_1(x)R_2'(\bar z) z_\gamma,& 0\le x\le \bar \theta, \\
 \displaystyle
\beta w(0)+R_1(x)R_2'(\bar z)z_\gamma+1, &  \bar \theta<x\le 1.
\end{cases}
\end{align}
 The computation of $w$ now reduces to uniquely solving $w(0)$.
By the expression of $w$ on $[0, \bar\theta]$,  we have
\begin{align}
w(\bar\theta)&= \beta \int_{\bar\theta}^1 w(y)Q_0(dy|\bar \theta) +
R_1(\bar\theta)R_2'(\bar z) z_\gamma\nonumber\\
&=\beta^2 w(0)+\beta+R_1(\bar \theta)R_2'(\bar z) z_\gamma   + \frac{\beta R_2' (\bar z)z_\gamma}{1-\bar\theta}\int_{\bar \theta}^1 R_1(y)dy\nonumber \\
&= \beta^2 w(0)+\beta+\bar\theta z_\gamma + \beta z_\gamma \frac{1+\bar\theta}{2}. \label{ws1}
\end{align}
For $x\in [0, \bar\theta]$,
we further write
\begin{align}
w(x)=\frac{\beta}{1-x} \int_x^1 w(y)dy+R_1(x)R_2'(\bar z) z_\gamma,\nonumber
\end{align}
and solve
\begin{align}
w(x)= (1-x)^{\beta-1} w(0)+z_\gamma x-\beta z_\gamma
\Big[\frac{(1-x)^{\beta-1}}{(1-\beta)(2-\beta)} -\frac{1}{1-\beta} +\frac{1-x}{2-\beta}\Big],
\nonumber
\end{align}
which further gives
\begin{align}\label{ws2}
w(\bar  \theta)=  (1-\bar \theta)^{\beta-1} w(0)+ z_\gamma\bar \theta -\beta z_\gamma
\Big[\frac{(1-\bar\theta)^{\beta-1}}{(1-\beta)(2- \beta)} -\frac{1}{1-\beta} +\frac{1-\bar\theta}{2-\beta}\Big].
\end{align}
 By \eqref{ws1}--\eqref{ws2},  we have
\begin{align}\label{csew0}
[\beta^{-1}(1-\bar \theta)^{\beta-1}-\beta]w(0)= 1+z_\gamma \Big(\frac{1+\bar\theta}{2} +\frac{(1-\bar\theta)^{\beta-1}}{ (1-\beta)(2-\beta)} +\frac{1-\bar\theta}{2-\beta} - \frac{1}{1-\beta}  \Big).
\end{align}

Now from \eqref{pinv} we have
\begin{align}
p(x)=\begin{cases}
 \displaystyle
\int_0^x \frac{1}{1-y} p(y) dy+\pi_0,& 0\le x<\bar \theta, \\
 \displaystyle
\int_0^{\bar\theta} \frac{1}{1-y}p(y) dy +\pi_0,&\bar \theta\le x\le 1,
\end{cases}\nonumber
\end{align}
which determines
\begin{align}
p(x) = \begin{cases}
  \displaystyle
\frac{\pi_0}{1-x}, & 0\le x<\bar\theta,\\
 \displaystyle
\frac{\pi_0}{1-\bar \theta}, & \bar\theta\le x\le 1,
\end{cases}\nonumber
\end{align}
where $\pi_0= \frac{1}{2-\ln (1-\bar\theta)}$. We determine the mean field
\begin{align}\label{unizb}
\bar z=\int_0^{\bar\theta} x p(x)dx +\int_{\bar\theta}^1 x p(x) dx = \pi_0
\big(\frac{1-\bar\theta}{2}-\ln (1-\bar\theta)\big).
\end{align}
We further obtain $\frac{dz}{d\gamma}$ at $\bar \gamma$ as
\begin{align}\label{cse3z}
z_\gamma = \frac{\ln (1-\bar \theta)-3 +\frac{4}{1-\bar \theta}}{2[2-\ln (1-\bar \theta)]^2} \theta_\gamma.
\end{align}
We note that a perturbation analysis directly based on the general case \eqref{pinv} is more complicated.

Now \eqref{bsiga}
 reduces to
\begin{align}
\Big[\frac{\beta}{1-\bar \theta}  \int_{\bar\theta}^1 \frac{\bar v(y)}{1-\bar\theta} dy -\frac{\beta \bar v(\bar\theta)}{1-\bar\theta}\Big]
\theta_\gamma = 1+\beta w(0)- \beta \int_{\bar\theta}^1 \frac{w(y)}{1-\bar\theta} dy.\nonumber
\end{align}
By the expression of $\bar v$ in \eqref{vzs} and $w$ in \eqref{wequni} at $\theta=\bar \theta$, we obtain $$
\frac{(1-\beta)\bar v(\bar\theta)-\bar\theta (c+\bar z)}{1-\bar \theta}\theta_\gamma
=1+\beta w(0)-w(\bar \theta)+\bar\theta z_\gamma. $$
Recalling \eqref{vb0sg} and \eqref{ws1}, we have
\begin{align} \label{cse3}
\frac{(1-\beta)[\beta \bar v(0)+\bar\gamma]-\beta \bar \theta (\bar z+c) }{1-\bar \theta}\theta_\gamma
- \beta (1-\beta )
w(0)+   \frac{1+\bar\theta}{2}\beta z_\gamma = 1-\beta.
\end{align}

By combining  \eqref{univ0}, \eqref{unith} and \eqref{unizb},  we have
\begin{align}
&\bar v(0)=[(1-\bar \theta)^{\beta-1}-\beta ]^{-1}\Big[\frac{ \beta (c+\bar z)
[(1-\bar \theta)^{\beta-1}-1]}{(1-\beta)(2-\beta)} -\frac{\beta  (c+\bar z)\bar\theta}{2-\beta}
+\bar \gamma\Big], \label{mfv1}  \\
&\bar \theta= \frac{2(1-\beta)(\beta  \bar v(0) +\bar \gamma) }{\beta (c+\bar z)}-1, \label{mfs1} \\
&\bar z=\frac{1}{2-\ln(1-\bar\theta)} \big(\frac{1-\bar \theta}{2}-\ln (1-\bar \theta)\big). \label{mfz1}
\end{align}
Next, combining \eqref{csew0}, \eqref{cse3z} and \eqref{cse3},  we obtain
\begin{align}
& \frac{(1-\beta)[\beta \bar v(0)+\bar\gamma]-\beta \bar \theta (\bar z+c) }{1-\bar \theta}\theta_\gamma
- \beta (1-\beta )
w(0)+   \frac{1+\bar\theta}{2}\beta z_\gamma = 1-\beta , \label{mfsen1} \\
&[\beta^{-1}(1-\bar \theta)^{\beta-1}-\beta]w(0)= 1+z_\gamma \Big(\frac{1+\bar\theta}{2} +\frac{(1-\bar\theta)^{\beta-1}}{ (1-\beta)(2-\beta)} +\frac{1-\bar\theta}{2-\beta} - \frac{1}{1-\beta}  \Big), \label{mfsen2}\\
&z_\gamma = \frac{\ln (1-\bar \theta)-3 +\frac{4}{1-\bar \theta}}{2[2-\ln (1-\bar \theta)]^2} \theta_\gamma. \label{mfsen3}
\end{align}
 After $(\bar v(0),\bar z, \bar \theta )$ has been determined from \eqref{mfv1}-\eqref{mfz1}, the above gives a linear  equation system with unknowns $w(0)$, $\theta_\gamma$ and $ z_\gamma$.
\begin{example}
We take $R_1(x)=x$ and $R_2(z)= 0.2+z$, $\bar \gamma=0.5$, $\beta=0.9$.\footnote{Corrected on Oct 10, 2020 by adding the value of $\beta$ and correcting the parameter in $R_2(z)$.}
We numerically solve \eqref{mfv1}-\eqref{mfz1} to obtain $\bar v(0)= 3.497854,\  \bar \theta =0.485162,\  \bar z =  0.345854$, and \eqref{mfsen1}-\eqref{mfsen3}
 to obtain  $ w(0)= 4.563055,\  \theta_\gamma= 1.162861,\
 z_\gamma=  0.336380$. The curves of $v(x)$ and $w(x)$ are displayed in Fig. 1, where $w$ has a discontinuity at $x=\bar \theta$ as discussed in Remark \ref{re:dis}. The positive value of $\theta_\gamma$ implies the threshold increases with $\gamma$, as asserted in Theorem \ref{theorem:mcs}.
\end{example}
\begin{figure}[t]
\begin{center}
\psfig{file=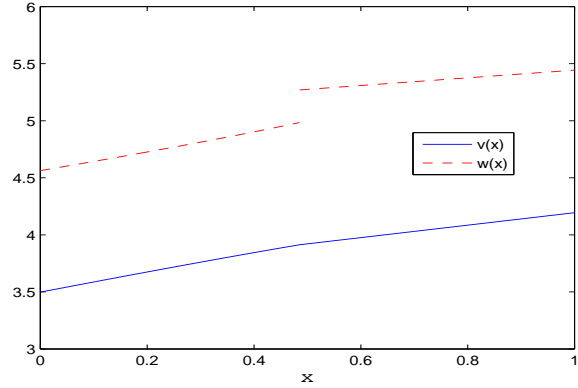, width=3.6in, height=2.2in}\\
\end{center}
\caption{ Value function $v$ and perturbation function $w$ } \label{fig1}
\end{figure}

\section{Conclusion}
\label{sec:con}

This paper considers mean field games in a framework of binary Markov decision processes (MDP) and establishes existence and uniqueness  of  stationary  equilibria. The resulting policy has a threshold structure.
We further analyze comparative statics to address the impact of parameter variations in the model.

For future research, there are some potentially interesting extensions. One may consider a heterogenous population and study the emergence of free-riders who care more about their own effort costs and have less incentive to
contribute to the common benefit of the population. Another modelling of a quite different nature involves negative externalities where other players' improvement brings more pressure on the player in question. For instance, this arises in competitions for   market share. The modelling and  analysis of the agent behavior will be of interest.

\section*{Appendix A: Preliminaries on Ergodicity }
\renewcommand{\theequation}{A.\arabic{equation}}
\setcounter{equation}{0}
\renewcommand{\thetheorem}{A.\arabic{theorem}}
\setcounter{theorem}{0}
\renewcommand{\thelemma}{A.\arabic{lemma}}
\setcounter{lemma}{0}
\renewcommand{\theremark}{A.\arabic{remark}}
\setcounter{remark}{0}

Assume (A3).
The next two lemmas
determine the limiting distribution of the state process  under threshold policies.

\begin{lemma}\label{lemma:limd}
i)
If $\theta=0$, then the distribution of $x_t^i$ remains to be  the dirac measure $\delta_0$ for all $t\ge 1$, for any $x_0^i$.

ii) If $\theta=1$ or $\theta=1^+$, the distribution of $x_t^i$ converges to the dirac measure $\delta_1$ weakly.
\end{lemma}

\begin{proof}
Part i) is obvious and part ii)  follows from (A3).  \qed
\end{proof}

Let $x_t^{i,\theta}$ denote the state process generated by the $\theta$-threshold policy  with  $\theta\in (0, 1)$, and let $P_\theta^t(x, \cdot)$ be the distribution of $x_t^{i,\theta}$
 given $x_0^{i,\theta}=x$.

\begin{lemma} \label{lemma:erg}
For $\theta\in (0, 1)$,  $\{x_t^{i,\theta}, t\ge 0\}$
is uniformly ergodic  with stationary probability distribution $\pi_\theta$, i.e.,
\begin{align}
\sup_{x\in \bS}\|P_\theta^t(x, \cdot)
-\pi_\theta\|_{\rm TV}\le K r^t, \label{ppit}
\end{align}
for some constants $K>0$ and $r\in (0, 1)$,
where $\|\cdot\|_{\rm TV}$ is the total variation norm of signed measures.
\end{lemma}

\proof The proof is similar to that of the ergodicity
theorem in \cite{HM16Chen}, which assumed (A3$^\prime$).
We use (A3)-iii) to estimate $r$. \qed

We  take $C_s=\{0\}$ as a small set and $\theta\in (0, 1)$. The $\theta$-threshold policy gives
\begin{align}
P(x_2^{i,\theta}=0|x_0^{i,\theta}=0) \ge\int_\theta^1 q(y|0)dy =:  \epsilon_0. \label{se0}
\end{align}
So for any Borel set  $B$,
$ P(x_2^{i,\theta}\in B|x_0^{i,\theta}=0)\ge \epsilon_0 \delta_0(B)$, where $\delta_0$ is the dirac measure. For $\theta'$ in a small neighborhood of $\theta$, we can ensure that the $\theta'$-threshold policy gives
\begin{align}
P(x_2^{i,\theta'}\in B|x_0^{i,\theta'}=0)\ge \frac{\epsilon_0}{2} \delta_0(B). \label{se0p}
\end{align}

\begin{lemma} \label{lemma:sipi}
Suppose $\theta, \theta'\in (0,1)$ for two threshold policies.
 Let the corresponding stationary distributions of the state process by  $\pi$ and $\pi'$.
Then
$$
\lim_{\theta'\to \theta}\|\pi'-\pi\|_{\rm TV}=0.
$$
\end{lemma}

\proof Fix $\theta\in (0, 1)$. By \eqref{se0p} and \cite{MT09},
there exist a  neighborhood
$I_0=(\theta- \kappa_0, \theta+\kappa_0)\subset (0, 1)$  and two constants $C$, $r\in (0, 1)$    such that for all
$\theta'\in I_0$,
$$
\|P^t_\theta(x, \cdot) -\pi \|_{\rm TV} \le C r^t, \quad \|P^t_{\theta'}(x, \cdot) -\pi' \|_{\rm TV} \le C r^t, \quad \forall x \in [0, 1].
$$
Subsequently,
$$
\|\pi'-\pi\|_{\rm TV}\le  \|P_{\theta'}^t(0, \cdot)-P_\theta^t(0, \cdot)  \|_{\rm TV}+
2Cr^t.
$$ For any given $\epsilon>0$, fix a large $k_0$ such that $2Cr^{k_0}\le \epsilon/2$. We  show for all $\theta'$ sufficiently close to $\theta$,
$$
\|P^{k_0}_{\theta'} (0, \cdot)-P^{k_0}_\theta(0, \cdot)\|_{\rm TV} \le \epsilon/2. $$
Given two probability measures $\mu_t$, $\mu_t'$, define the  probability measures $\mu_{t+1}$ and $\mu_{t+1}'$,
$$\mu_{t+1}(B)=\int_{\bf S} P_\theta(y, B)\mu_t (dy), \quad
\mu'_{t+1}(B)=\int_{\bf S} P_{\theta'}(y, B)\mu_t' (dy),
$$
for Borel set $B\subset [0, 1]$.
Then
\begin{align*}
|\mu_{t+1}(B)-\mu_{t+1}'(B) | \le \ & |\int_{\bf S}
P_\theta(y, B)\mu_t (dy)-\int_{\bf S} P_{\theta'}(y, B)\mu_t (dy)|\\
&+ | \int_{\bf S} P_{\theta'}(y, B)\mu_t (dy)  - \int_{\bf S} P_{\theta'}(y, B)\mu_t' (dy)|\\
& =: D_1+D_2.
 \end{align*}
We have
\begin{align*}
D_2= \Big| \int_{\bf S} P_{\theta'}(y, B)\mu_t (dy)  - \int_{\bf S} P_{\theta'}(y, B)\mu_t' (dy)\Big|
\le 2\|\mu_t-\mu'_t\|_{\rm TV}.
\end{align*}
Denote $\underline \theta=\min\{\theta, \theta'\}$ and $\overline \theta= \max\{\theta, \theta'\}$.
Then
\begin{align*}
D_1&= \Big|-\int_{[\underline\theta, \overline\theta)}Q_0(B|y) \mu_t(dy) +1_B(0)
\mu_t ([\underline\theta, \overline\theta))\Big|
\le \mu_t([\underline\theta, \overline\theta)).
\end{align*}
Setting $\mu_0=\mu_0'=\delta_0$, then $\mu_t= P_\theta^t(0, \cdot)$, $\mu_t'= P_{\theta'}^t(0, \cdot)$.   Hence,
\begin{align}
|P_{\theta'}^{t+1}(0, B)-P_\theta^{t+1} (0, B) |
\le\ & 2\|P_{\theta'}^{t}(0, \cdot)-P_\theta^{t} (0, \cdot) \|_{\rm TV}  + P^t_\theta(0, [\underline\theta, \overline\theta)),
\end{align}
which implies
\begin{align}
\|P_{\theta'}^{t+1}(0, \cdot)-P_\theta^{t+1} (0, \cdot) \|_{\rm TV}
\le\ & 4\|P_{\theta'}^{t}(0, \cdot)-P_\theta^{t} (0, \cdot) \|_{\rm TV}  + 2P^t_\theta(0, [\theta, \theta')). \label{psisi}
\end{align}
For $\mu_0=\mu_0'=\delta_0$, we have $P^1_\theta(0, \cdot)=P^1_{\theta'}(0,\cdot)$. It is clear from \eqref{psisi} and Lemma \ref{lemma:pi0}
 that for each $t\ge 1$,
$$
\lim_{\theta'\to \theta}\|P_{\theta'}^t (0, \cdot)- P_\theta^t (0, \cdot) \|_{\rm TV} =0, \quad \lim_{\theta'\to \theta}  P^t_\theta(0, [\underline\theta, \overline\theta))=0.
$$
Therefore, for the fixed $k_0$, there exists $\delta>0$ such
that for all $\theta'$ satisfying  $|\theta'-\theta|<\delta$,
$ \|P_{\theta'}^{k_0} (0, \cdot)- P_\theta^{k_0} (0, \cdot) \|_{\rm TV} <
\frac{\epsilon}{2}$ and   $\|\pi'-\pi\|_{\rm TV}\le \epsilon$.   The lemma follows. \qed

\section*{Appendix B:  Cycle Average of A Regenerative Process}
\renewcommand{\theequation}{B.\arabic{equation}}
\setcounter{equation}{0}
\renewcommand{\thetheorem}{B.\arabic{theorem}}
\setcounter{theorem}{0}
\renewcommand{\thelemma}{B.\arabic{lemma}}
\setcounter{lemma}{0}
\renewcommand{\theremark}{B.\arabic{remark}}
\setcounter{remark}{0}

Let $0<r<r'<1$.
Consider a Markov process $\{Y_t, t\ge 0\}$ with  state space $[0, 1]$
  and transition kernel $Q_Y(\cdot|y)$ which satisfies
 $Q_Y([y,1]|y)=1$ for any $y\in [0, 1]$ and is stochastically increasing. Suppose $Y_0\equiv y_0<r$.
Define the stopping times
$$
\tau=\inf \{t|Y_t\ge r\},\quad  \tau'=\inf\{t|Y_t\ge r'\}.
$$

\begin{lemma} \label{lemma:yav}
If $E\tau<\infty $, then  $ E\sum_{t=0}^\tau Y_t<\infty$ and
\begin{align}
\frac{E\sum_{t=0}^\tau Y_t}{1+E\tau }
 =\frac{EY_0+ EY_1+ \sum_{k=1}^\infty E(Y_{k+1} 1_{\{Y_k<r\}})}{2+
 \sum_{k=1 }^\infty P(Y_k<r)  }.
 \end{align}
\end{lemma}

\proof Since $0\le Y_t\le 1 $ w.p. 1, $ E\sum_{t=0}^\tau Y_t\le 1+E\tau$.
 It is clear that $\{\tau\ge k\}= \{Y_{k-1}<r\}$ for
 $k\ge 1$.
 We have
\begin{align}
E\tau &= \sum_{k=1}^\infty P(\tau\ge k)=1+ \sum_{k=1}^\infty P(Y_{k}<r), \label{etau}
\end{align}
and
\begin{align*}
E\sum_{t=0}^\tau Y_t&= E \sum_{k=1}^\infty \left(\sum_{t=0}^kY_t\right)1_{\{\tau=k\}}  \\
&  =EY_0+EY_1+ \sum_{k=2}^\infty E(Y_{k} 1_{\{\tau\ge k\}})\\
& =EY_0+ EY_1+\sum_{k=1}^\infty E(Y_{k+1} 1_{\{Y_k<r\}}).
\end{align*}
The lemma follows. \qed

\begin{lemma} \label{lemma:tatap}
Assume $E\tau'<\infty$.
We have
\begin{align}
\frac{E\sum_{t=0}^\tau Y_t}{1+E\tau } \le
\frac{E\sum_{t=0}^{\tau'} Y_t}{1+E\tau' }. \label{yly}
\end{align}
\end{lemma}

\proof  $E\tau<\infty$ since $\tau\le \tau'$ w.p.1.
 For $k\ge 1$,  denote
\begin{align*}
&p_k= P(Y_k<r), \quad \eta_k=P(r\le Y_k< r'),\\
& m_k=
E(Y_{k+1} 1_{\{Y_k<r\}}),\quad
\Delta_k= E(Y_{k+1} 1_{\{r\le Y_k<r'\}}).
\end{align*}
By Lemma \ref{lemma:yav},
\begin{align*}
&\frac{E\sum_{t=0}^\tau Y_t}{1+E\tau }=\frac{EY_0+EY_1+
\sum_{k=1}^\infty m_k}{2+\sum_{k=1}^\infty  p_k},\\
& \frac{E\sum_{t=0}^{\tau'} Y_t}{1+E\tau' }=
\frac{EY_0+EY_1+ \sum_{k=1}^\infty (m_k+\Delta_k)}{2+
\sum_{k=1}^\infty  (p_k+\eta_k)}.
\end{align*}
So \eqref{yly} is equivalent to
\begin{align}
(EY_0+EY_1+\sum_{k=1}^\infty m_k)(\sum_{k=1}^\infty \eta_k)  \le (\sum_{k=1}^\infty \Delta_k) (2+\sum_{k=1}^\infty  p_k). \label{ymdp}
\end{align}
By the stochastic monotonicity of $Q_Y$, we have
\begin{align*}
E[Y_{k+1} 1_{\{Y_k<r\}} |Y_k] &= 1_{\{Y_k<r\}} \int_0^1 y Q_Y(dy|Y_k)\\
&\le  1_{\{Y_k<r\}}\int_0^1 y Q_Y(dy|r)=: c_r1_{\{Y_k<r\}}.
\end{align*}
Note that
\begin{align}
c_r=\int_{y\ge r} yQ_Y(dy|r)\ge r. \label{crr}
\end{align}
Moreover,
\begin{align*}
E[Y_{k+1} 1_{\{r\le Y_k<r'\}}|Y_k]&=1_{\{r\le Y_k<r'\}}\int_0^1 y
 Q_Y(dy|Y_k) \\
& \ge c_r 1_{\{r\le Y_k<r'\}} .
\end{align*}
It follows that
\begin{align}
&m_k=E[Y_{k+1} 1_{\{Y_k<r\}}] \le c_r p_k, \quad \Delta_k=E[Y_{k+1} 1_{\{r\le Y_k<r'\}}]\ge c_r \eta_k . \label{mdel}
\end{align}
Since $Y_0=y_0<r$,
$$ E[Y_1|Y_0]= \int_0^1 y Q_Y(dy|Y_0)\le c_r.$$
Hence, $E(Y_0+Y_1)\le r+c_r $.  By \eqref{mdel} and \eqref{crr},
 \begin{align*}
&  (EY_0+EY_1+\sum_{k=1}^\infty m_k)(\sum_{k=1}^\infty \eta_k)  - (\sum_{k=1}^\infty \Delta_k) (2+\sum_{k=1}^\infty  p_k)\\
\le &  (r+c_r+c_r \sum_{k=1}^\infty p_k) (\sum_{k=1}^\infty \eta_k) - c_r (\sum_{k=1}^\infty \eta_k)
(2+\sum_{k=1}^\infty p_k)  \\
=& (r-c_r) \sum_{k=1}^\infty \eta_k
\le 0,
\end{align*}
which establishes  \eqref{ymdp}. \qed

\begin{remark}\label{rem:b1}
If for each $y\in [0, 1)$,  $Q_Y(dx|y)$ has probability density function $q_Y(x|y)>0$ for $x\in (y, 1)$, then $c_r>r$ and $\eta_k>0$ for all $k\ge 1$. In this case, a strict inequality holds for \eqref{yly}. \qed
\end{remark}

\section*{Appendix C}
\label{sec:proofz}

\renewcommand{\theequation}{C.\arabic{equation}}
\setcounter{equation}{0}
\renewcommand{\thetheorem}{C.\arabic{theorem}}
\setcounter{theorem}{0}
\renewcommand{\thelemma}{C.\arabic{lemma}}
\setcounter{lemma}{0}
\renewcommand{\thesubsection}{C.\arabic{subsection}}
\setcounter{subsection}{0}
\renewcommand{\theremark}{C.\arabic{remark}}
\setcounter{remark}{0}

We assume (A3).
Let $\{x_t^{i,\theta}, t\ge 0\}$ be the Markov chain generated by  a
$\theta$-threshold policy with  $0<\theta<1$, where $x_0^{i,\theta}$ is given.
By Lemma \ref{lemma:erg},  $\{x_t^{i,\theta}, t\ge 0\}$ is  ergodic.
 We next define an auxiliary  Markov chain $\{Y_t, t\ge 0\}$ with $Y_0=0$ and the same transition kernel as $x_t^{i,\theta}$.
Denote
$
S_t =\sum_{i=0}^t Y_i
$ for $t\ge 0$.
Define $\tau=\inf\{t|Y_t \ge \theta\}$.

\begin{lemma} \label{lemma:YS}
We have
\begin{align}
\lim_{k\to \infty}\frac{1}{k} \sum_{t=0}^{k-1} Y_t= \frac{ES_\tau}
{ 1+E\tau} \qquad{\rm  w.p.1}.  \label{lrun}
\end{align}
\end{lemma}

\proof By (A3), we can show $E\tau <\infty$.   Since $\{Y_t, t\ge 0\}$ has the same transition probability kernel as $\{x_t^{i, \theta}, t\ge 0 \}$, it is ergodic, and therefore the left hand side of \eqref{lrun} has a constant limit   w.p.1.
Define $T_0=0$ and $T_n$ as the time for $\{Y_t, t\ge 0\}$ to return to state 0 for the $n$th time. So $T_1=\tau+1$.  Define  $B_n= \sum_{t= T_{n-1}}^{T_n-1} Y_t$ for $n\ge 1$. We observe that $\{Y_t, t\ge 0\}$ is a regenerative process
(see e.g. \cite{A03,SW93} and \cite[Theorem 4]{AR14})
 with regeneration times $\{T_n, n\ge 1\}$ and that $\{B_n, n\ge 1\}$ is a sequence of i.i.d. random variables. Note that $B_1=S_\tau $ is the sum of $\tau+1$ terms. By the strong  law of large numbers for regenerative processes \cite[pp. 177]{A03},
 the lemma follows. \qed

Suppose $0<\theta<\theta'<1$. Then there exist two constants $C_\theta, C_{\theta'}$ such that
$$
\lim_{k\to \infty }  \frac{1}{k}\sum_{t=0}^{k-1}x_t^{i,\theta}=C_\theta, \qquad \lim_{k\to \infty }  \frac{1}{k}\sum_{t=0}^{k-1}x_t^{i, \theta'}
= C_{\theta'},  \quad \mbox{w.p.1.}
$$

\begin{lemma} \label{lemma:bb}
We have $C_\theta\le C_{\theta'}$.
\end{lemma}

\proof Due to the ergodicity of the Markov chain, $C_\theta$ (resp., $C_{\theta'}$) does not depend on   $x_0^{i, \theta}$ (resp., $x_0^{i, \theta'}$). Therefore, $\lim_{k\to \infty}\frac{1}{k} \sum_{t=0}^{k-1} Y_t=C_\theta$ w.p.1.
The lemma follows from     Lemmas  \ref{lemma:YS} and \ref{lemma:tatap}.~\qed

\section*{Appendix D: An Auxiliary MDP}
\renewcommand{\theequation}{D.\arabic{equation}}
\setcounter{equation}{0}
\renewcommand{\thetheorem}{D.\arabic{theorem}}
\setcounter{theorem}{0}
\renewcommand{\thelemma}{D.\arabic{lemma}}
\setcounter{lemma}{0}
\renewcommand{\theremark}{D.\arabic{remark}}
\setcounter{remark}{0}

Assume (A3).
This appendix introduces an auxiliary control problem
to show the effect of the effort cost on the threshold parameter of the optimal policy.
The state and control processes
$\{(x_t^i, a_t^i), t\ge 0 \}$  are specified by  \eqref{xa0}-\eqref{xa1}. The cost has the form
\begin{align}
J_i^r=E\sum_{t=0}^\infty \rho^t \big(R_1(x^i_t)+  r 1_{\{a_t^i=a_1\}} \big),
\label{jir}
\end{align}
where $R_1$ is continuous and strictly increasing on $[0, 1]$ and $\rho \in (0,1)$, $r\in (0, \infty)$.
Let $r$ take two different values  $0<\gamma_1<\gamma_2$ and write the
corresponding dynamic programming equation
\begin{align}
v_l(x) =\min \left\{\rho \int_0^1 v_l (y) Q_0(dy|x)+R_1(x), \quad \rho v_l(0) +R_1(x)+\gamma_l   \right\} , \quad l=1,2,\ x\in \bS.\label{vlr12}
\end{align}

By the method in proving Lemma \ref{lemma:LV}, it can be shown that there exists a unique solution
$v_l\in C([0,1], \mathbb{R})$ and that the optimal policy $a^{i,l}(x)$ is a threshold policy.
If
$ \rho \int_0^1 v_l (y) Q_0(dy|1)<\rho v_l(0)+\gamma_l,
$
$a^{i,l}(x)\equiv  a_0$, and we follow the notation in Section \ref{sec:br}
to  denote the threshold  $\theta_l=1^+$. Otherwise, $a^{i,l}(x)$ is a
$\theta_l$-threshold policy with  $\theta_l\in [0,1]$, i.e., $a^{i,l}(x)= a_1$ if $x\ge \theta_l$, and $a^{i,l}(x)=a_0$ if $x<\theta_l$.

\begin{lemma}\label{lemma:the12}
 If $\theta_1\in (0, 1)$, $\theta_2\ne \theta_1$.
\end{lemma}

\proof
We prove by contradiction. Suppose for some $\theta\in (0,1)$,
\begin{align}
\theta_1=\theta_2=\theta.  \label{th12}
\end{align}
Under   \eqref{th12},
 the resulting optimal policy leads to  the representation (see e.g. \cite[pp. 22]{HL93})
$$
v_l(x)= E\sum_{t=0}^\infty \rho^t \left[R_1( x_t^{i})+\gamma_l 1_{\{ a_t^i=a_1\}}\right], \quad l=1,2,
$$
where $\{ x_t^i, t\ge 0\}$ is generated by the $\theta$-threshold policy $a_t^i( x_t^i)$ and $x_0^i=x$. Denote $\delta_{21}=\gamma_2-\gamma_1$.

For  fixed $x\ge \theta$ and $ x_0^i=x$, denote the resulting optimal state and control processes by $\{(\hat x_t^i, \hat a_t^i), t\ge 0  \}$.
Then $\hat a_0^i =a_1$ w.p.1.,  and
$$
v_2(x)-v_1(x)= \delta_{21} +\delta_{21} E\sum_{t=1}^\infty \rho^t1_{\{ \hat a_t^i=a_{1}\}}, \qquad x\ge \theta.
 $$

Next consider  $x_0^i=0$ and denote the optimal state and control processes by
$\{(\check{x}_t^i, \check{a}_t^i), t\ge 0  \}$.
Then
$$
v_2(0)-v_1(0)= \delta_{21} E\sum_{t=0}^\infty \rho^t1_{\{ \check{a}_t^i=a_{1}\}}
=: \Delta.
$$

It is clear that $\hat x_1^i=0$ w.p.1. By the optimality principle,
$\{(\hat x_t^i, \hat a_t^i), t\ge 1  \}$ may be interpreted as the optimal state and control processes of the MDP with initial
state 0 at $t=1$. Hence the two processes $\{(\hat x_t^i, \hat a_t^i), t\ge 1  \}$ and $\{(\check{x}_t^i, \check{a}_t^i), t\ge 0  \}$, where ${\check x}_0^i=0$, have the same finite dimensional distributions. In particular, $\hat a_{t+1}^i$ and $\check{a}_t^i$ have the same distribution for $t\ge 0$.
Therefore,
$$
E\sum_{t=1}^\infty \rho^{t-1}1_{\{ \hat a_t^i=a_{1}\}}
=
E\sum_{t=0}^\infty \rho^t1_{\{ \check{a}_t^i=a_{1}\}}.
$$
It follows that
\begin{align}
v_2(x)-v_1(x)= \delta_{21}+\rho \Delta, \qquad \forall x\ge  \theta. \label{v12d}
\end{align}
Combining
  \eqref{vlr12} and \eqref{th12} gives
 \begin{align*}
\rho \int_0^1 v_l(y) Q_0(dy|\theta)=  \rho v_l(0)+\gamma_l  ,
 \qquad l=1, 2,
 \end{align*}
 which implies
 \begin{align}
\rho \int_0^1 [v_2(x)-v_1(x)] Q_0(dx|\theta)  =\delta_{21}+ \rho \Delta.  \label{v21del}
 \end{align}
By  $Q_0([0, \theta)|\theta)=0$ and \eqref{v12d},  \eqref{v21del} further yields
 $\rho( \delta_{21}+ \rho \Delta) =  \delta_{21}+\rho \Delta$,
 which is impossible since $0<\rho<1$ and $\delta_{21}+\rho \Delta>0$.
 Therefore, \eqref{th12} does not hold. This completes the proof.
  \qed

For the MDP with cost
\eqref{jir},
we continue to analyze the  dynamic programming equation
\begin{align}\label{Vstab}
v_r(x) =  \min \Big[\rho \int_0^1 v_r( y) Q_0(dy|x) + R_1(x),  \quad \rho v_r(0) + R_1(x)+ r\Big].
\end{align}
For each fixed $r \in (0, \infty )$, we obtain the optimal policy as a threshold policy with threshold parameter $\theta(r)$.
 By evaluating the cost \eqref{jir}   associated with  the two policies
 $a_t^i(x_t^i)\equiv a_0$ and $a_t^i(x_t^i)\equiv a_1$, respectively,  we have the prior estimate
\begin{align}
 v_r(x)\le \min\left\{ \frac{R_1(1)}{1-\rho },\ R_1(x)+\frac{r +\rho R_1(0) }{1-\rho}   \right\}. \label{rvr}
\end{align}
On the other hand, let $\{x_t^i, t\ge 0\}$ with $x_0^i=x$ be generated by any fixed Markov policy. Then
\begin{align}
E\sum_{t=0}^\infty \rho^t (R_1(x_t^i)+ r 1_{\{a_t^i=a_1\}})\ge R_1(x) +
\sum_{t=1}^\infty \rho^t R_1(0), \nonumber
\end{align}
which implies
\begin{align}
v_r(x)\ge R_1(x)+\frac{\rho R_1(0)  }{1-\rho }. \label{vrge}
\end{align}

If $r>\frac{\rho R_1(1)}{1-\rho }$, it follows from \eqref{rvr} that
\begin{align}
\rho \int_0^1 v_r( y) Q_0(dy|x) < \rho v_r(0)+r,   \quad \forall x,  \label{th1}
\end{align}
i.e., $\theta(r)=1^+$.

\begin{lemma}\label{lemma:th0}
There exists $\delta>0$ such that for all $0<r<\delta$,
\begin{align}
\rho \int_0^1 v_r( y) Q_0(dy|x)> \rho v_r(0)+r, \quad \forall x,  \label{th0}
\end{align}
and so $\theta(r)=0$.
\end{lemma}

\proof By \eqref{vrge},
\begin{align}
\rho \int_0^1 v_r( y) Q_0(dy|x) &\ge \rho\int_0^1 R_1(y) Q_0(dy|x)
+\frac{\rho^2 R_1(0)}{1-\rho } \nonumber \\
 & \ge \rho\int_0^1 R_1(y) Q_0(dy|0)
+\frac{\rho^2 R_1(0)}{1-\rho }  , \nonumber
\end{align}
and \eqref{rvr} gives
$$
\rho v_r(0)+r \le
\frac{ \rho R_1(0)}{1-\rho }+
\frac{ r }{1-\rho }.
$$
Since $R_1(x)$ is strictly increasing,
$$
C_{R_1}: =\int_0^1 R_1(y) Q_0(dy|0)- R_1(0)>0.
$$
And we have
$$
\rho \int_0^1 v_r( y) Q_0(dy|x)- (\rho v_r(0)+r) \ge \rho C_{R_1}- \frac{r}{1-\rho}.
$$
It suffices to take $\delta =\rho (1-\rho) C_{R_1}.$ \qed

Define the nonempty sets
$$
{\cal R}_{a_0}=\{r>0| \eqref{th1} \mbox{ hods}\}, \quad {\cal R}_{a_1}=\{r>0|  \eqref{th0} \mbox{ holds}\}  .
$$

\begin{remark}\label{reC:int}
We have $(\frac{\rho R_1(1)}{1-\rho }, \infty   )\subset {\cal R}_{a_0} $ and  $(0, \delta)\subset {\cal R}_{a_1} $.
\end{remark}

\begin{lemma} \label{lemma:connect}
 Let $(r, v_r)$ be the parameter and the associated solution in \eqref{Vstab}.

i) If $r>0$ satisfies
\begin{align}
\rho \int_0^1 v_r(y) Q_0(dy|x) \le \rho v_r(0) +r, \quad \forall x,
\label{rth1}
\end{align}
  then  any $r'>r$ is in ${\cal R}_{a_0}$.

ii) If $r>0$ satisfies
\begin{align}
\rho \int_0^1 v_r(y) Q_0(dy|x) \ge \rho v_r(0) +r,\quad \forall x,
\label{rth0}
\end{align}
then any  $r'\in (0,r)$ is in ${\cal R}_{a_1}$.
\end{lemma}

\proof i) For $r'>r$, $v_{r'}$ is uniquely solved from \eqref{Vstab} with $r'$ in place of $r$. We can use \eqref{rth1} to  verify
$$
v_r(x)= \min \left[\rho \int_0^1 v_r(y) Q_0(dy|x) +R_1(x) , \quad \rho v_r(0) +R_1(x)+r'\right].
$$
Hence $v_{r'}=v_r$ for all $x\in [0,1]$. It follows that
$\rho \int_0^1 v_{r'}(y) Q_0(dy|x) <  \rho v_{r'}(0)+r'$ for all $x$. Hence  $r'\in {\cal R}_{a_0}$.

ii) By \eqref{Vstab} and  \eqref{rth0},
 $v_r(0) = \frac{ R_1(0)+r}{1-\rho } $,  
 and subsequently,
\begin{align*}
 v_r(x) &= \rho v_r(0) +R_1(x) +r
     = \frac{\rho R_1(0)+r}{1-\rho }  +R_1(x).  
\end{align*}
By substituting $v_r(0)$ and $v_r(x)$ into \eqref{rth0}, we obtain
\begin{align}
\rho R_1(0) +r \le \rho \int_0^1 R_1(y) Q_0(dy|x), \quad \forall x.
\label{rhorrle}
\end{align}
Now for $0<r'<r$, we construct $v_{r'}(x)$, as a candidate solution to
\eqref{Vstab} with $r$ replaced by $r'$,    to satisfy
\begin{align}
v_{r'}(0)= \rho v_{r'}(0)+R_1(0)+r', \quad
v_{r'}(x) = \rho v_{r'}(0) +R_1(x) +r', \label{defvp}
\end{align}
which gives
\begin{align}
v_{r'}(x) =\frac{\rho R_1(0)+r'  }{1-\rho }  +R_1(x).
\label{ss}
\end{align}

We  show that $v_{r'}(x)$ in \eqref{ss} satisfies
\begin{align}
\rho v_{r'}(0) +r' <\rho \int_0^1 v_{r'}(y) Q_0(dy|x), \quad \forall x,
\label{rovrp}
\end{align}
which is equivalent to
$
\rho R_1(0) +r' <\rho \int_0^1 R_1(y) Q_0(dy|x)$ for all $ x,
 $
which in turn follows from \eqref{rhorrle}.
By \eqref{defvp} and  \eqref{rovrp},  $v_{r'}$ indeed satisfies
\eqref{Vstab} with $r$ replaced by $r'$.  So $r'\in {\cal R}_{a_1}$. \qed

Further define
$$
\underline{r} =\sup {\cal R}_{a_1}, \quad \overline r= \inf  {\cal R}_{a_0}.
$$

\begin{lemma} \label{lemma:interv}
i) $\underline r$ satisfies
$\rho \int_0^1 v_{\underline r}(y) Q_0(dy|0) = \rho v_{\underline r}(0) +
\underline {r}$,
and $\theta(\underline r)=0$.

ii)
$\overline r$ satisfies
$\rho \int_0^1 v_{\overline r}(y) Q_0(dy|1) =\rho v_{\overline r}(1) = \rho v_{\overline r}(0) +
\overline {r}$,
and $\theta(\overline r)=1$.

iii)
We have  $0<\underline r<\overline r<\infty$.

iv) The threshold $\theta(r)$ as a function of $r\in (0, \infty)$ is continuous and strictly increasing on $[\underline r, \overline r]$.
\end{lemma}

\proof i)-ii) By Lemmas \ref{lemma:th0} and
\ref{lemma:connect}, we have
$ 0<\underline r\le \infty$ and $0\le  \overline r<\infty$.
Assume $\underline r=\infty$; then ${\cal R}_{a_1} =(0,\infty)$ giving
 ${\cal R}_{a_0}=\emptyset$, a contradiction. So $0<\underline r<\infty$.
For $\delta >0$ in Lemma \ref{lemma:th0}, we have $(0, \delta)\subset
{\cal R}_{a_1}$. Therefore, $0<\bar r <\infty$.
Note that $v_r$ depends on the parameter $r$ continuously, i.e.,
 $\lim_{|r'-r|\to 0}\sup_x|v_{r'}(x)-v_r(x)|=0$.
Hence
$$
\rho \int_0^1 v_{\underline r}(y) Q_0(dy|0) \ge  \rho
v_{\underline r}(0) +
\underline {r}.
$$
Now assume
\begin{align}
\rho \int_0^1 v_{\underline r}(y) Q_0(dy|0) >  \rho v_{\underline r}(0) +
\underline {r}. \label{rhor16}
\end{align}
Then there exists a sufficiently small $\epsilon>0$ such that \eqref{rhor16} still holds when
$(\underline r+\epsilon , v_{\underline r+\epsilon}   )$ replaces
$({\underline r} , v_{\underline r} )$; since $g(x)=\int_0^1
 v_{\underline r+\epsilon}(y) Q_0(dy|x) $ is increasing in $x$,  then $\underline r+\epsilon \in {\cal R}_{a_1}$, which is impossible. Hence \eqref{rhor16} does not hold, and this proves i). ii) can be shown in a similar manner.

To show iii),
 assume
\begin{align}
0< \overline r <\underline r < \infty.  \label{rrassum}
\end{align}
Then, recalling Remark \ref{reC:int}, there exist $r'\in {\cal R}_{a_0} $ and $r''\in {\cal R}_{a_1} $ such that
$$
0< \overline r<r'<r''<\underline r< \infty.
$$
By Lemma \ref{lemma:connect}-i), $r''\in {\cal  R}_{a_0}$, and
then $r''\in {\cal R}_{a_0}\cap {\cal  R}_{a_1}=\emptyset$, which is impossible. Therefore, \eqref{rrassum} does not hold and we conclude $0<\underline r\le \overline r<\infty$.
   We further assume $\underline r=\overline r$. Then
i)-ii) would imply
$\int_0^1  v_{\underline r}(y) Q_0(dy|0) = v_{\underline r}(1)$,
which is impossible since $v_{\underline r}$ is  strictly increasing
on $[0, 1]$ and (A3) holds. This proves iii).

iv) By the definition of $\underline r$ and $\overline r$, it can be shown using \eqref{Vstab} that
$\theta(r)\in (0, 1)$  for $r\in (\underline r, \overline r)$.
By the continuous dependence of the function $v_r(\cdot)$ on $r$ and the method of proving \cite[Lemma 10]{HM16Chen}, we can show  the continuity of $\theta(r)$ on $(0,1)$, and further show $\lim_{r\to \underline r^+}\theta(r)=0$ and $\lim_{r\to \overline{r}^-} \theta(r) =1$. So $\theta(r)$ is continuous on $[\underline r, \overline r]$.
If $\theta(r)$ were not strictly increasing on $[\underline r, \overline r]$,  there would exist
 $\underline r<r_1<r_2<\overline r$ such that
 \begin{align} \theta(r_1)\ge \theta(r_2). \label{s1ges2}
 \end{align}
  If $\theta (r_1)>\theta (r_2)$ in \eqref{s1ges2}, by the continuity of $\theta(r)$,
$\theta(\underline r)=0$, $\theta(\overline r)=1$, and the intermediate value theorem we may find $r'\in (\underline r, r_1)$ such that  $\theta( r_1')=\theta( r_2)$.
Next, we replace $r_1$ by $r_1'$.
Thus if $\theta(r)$ is not strictly increasing,  we may find $r_1<r_2$ from $(\underline r, \overline r)$ such that $\theta (r_1)=
\theta( r_2)\in (0, 1)$, which is a contradiction to Lemma \ref{lemma:the12}.
 This proves iv).
 \qed

\begin{remark}
By Lemmas \ref{lemma:connect} and \ref{lemma:interv},
${\cal R}_{a_1}= (0, \underline r)$ and
${\cal R}_{a_0}= (\overline r, \infty)$.
\end{remark}

\section*{Acknowledgement}
We would like to thank Aditya Mahajan for helpful discussions.

%
%

\begin{thebibliography}{99.}%
%
%

\bibitem{AJ13}
Acemoglu, D., Jensen, M.K.: Aggregate comparative statics. Games and Economic Behavior \textbf{81}, 27-49 (2013)


\bibitem{AJ15}
 Acemoglu, D., Jensen, M.K.: Robust comparative statics in large
dynamic economies. Journal of Political Economy \textbf{123}, 587-640 (2015)

\bibitem{AJW15}
Adlakha, S., Johari, R., Weintraub, G.Y.: Equilibria of dynamic
games with many players: Existence, approximation, and market
structure. J. Econ. Theory \textbf{156}, 269-316 (2015)



\bibitem{AS95}
Altman, E., Stidham, S.: Optimality of monotonic policies for two-action Markovian decision processes, with applications to control of queues with delayed information. Queueing Systems \textbf{21}, 267-291 (1995)


\bibitem{A96}
Amir R.: Sensitivity analysis of multisector optimal
economic dynamics. Journal of Mathematical Economics \textbf{25}, 123-141 (1996)


\bibitem{A03}
Asmussen, S.: Applied Probability and Queues, 2nd edn.
Springer, New York (2003)


\bibitem{AR14}
Athreya, K.B., Roy, V.: When is a Markov chain regenerative? Statistics and Probability Letters \textbf{84}, 22-26 (2014)




\bibitem{B13}
 Babichenko, Y.: Best-reply dynamics in large binary-choice anonymous games. Games and Economic Behavior \textbf{81}, 130-144 (2013)

\bibitem{B12}
 Bardi, M.: Explicit solutions of some linear-quadratic mean field games. Netw. Heterogeneous
Media \textbf{7}, 243-261 (2012)

\bibitem{BR11}
Bauerle, N., Rieder, U.: Markov Decision Processes with Applications to Finance.
Springer, Berlin (2011)

\bibitem{B85}
Becker R. A.:
Comparative dynamics in aggregate models of optimal capital accumulation. Quarterly Journal of Economics \textbf{100}, 1235-1256 (1985)




\bibitem{BFY13}
Bensoussan, A., Frehse, J., Yam, P.: Mean Field Games and Mean
Field Type Control Theory.  Springer, New York (2013)





\bibitem{B15}
 Biswas, A.: Mean field games with ergodic cost for discrete time Markov processes,
preprint, arXiv:1510.08968, 2015.




\bibitem{BS00}
Bonnans, J.F., Shapiro, A.: Perturbation Analysis of Optimization Problems. Springer-Verlag, New York (2000)


\bibitem{BD01}
 Brock, W.A., Durlauf, S. N.: Discrete choice with social interactions. Rev. Econ. Studies \textbf{68}, 235-260 (2001)





\bibitem{C14}
Caines, P.E.: Mean field games. In: Samad, T., Baillieul, J.  (eds.) Encyclopedia of Systems and Control. Springer-Verlag, Berlin (2014)

\bibitem{CHM17}
 Caines, P.E.,  Huang, M., Malham\'e, R.P.: Mean Field Games. In:  Basar, T.,  Zaccour, G. (eds.)
    Handbook of Dynamic
Game Theory, pp. 345-372,  Springer, Berlin  (2017)

\bibitem{C12}
Cardaliaguet, P.: Notes on mean field games, University of Paris, Dauphine (2012)

\bibitem{CD18}
Carmona R.,  Delarue, F.: Probabilistic Theory of Mean Field Games with Applications,
vol I and II.  Springer, Cham (2018)


\bibitem{D63}
Dorato, P.: On sensitivity in optimal control systems. IEEE Transactions on Automatic Control \textbf{8}, 256-257 (1963)



\bibitem{FV97}
Filar, J.A., Vrieze, K.: Competitive Markov
Decision Processes. Springer, New York (1997)


\bibitem{GMS10}
Gomes, D. A.,  Mohr, J., Souza, R.R.: Discrete time, finite state space mean field games. J. Math. Pures Appl. \textbf{93} 308-328, (2010)



\bibitem{HL93}
Hernandez-Lerma, O.: Adaptive Markov Control Processes. Springer-Verlag, New York (1989)

\bibitem{H39}
Hicks, J. R.:  Value and Capital.  Clarendon Press, Oxford
 (1939)




\bibitem{HCM03}
Huang, M., Caines, P.E., Malham\'e, R.P.: Individual and mass
behaviour in large population stochastic wireless power control
problems: Centralized and Nash equilibrium solutions. Proc.
42nd IEEE Conference on Decision and Control, pp. 98-105,  Maui, HI (2003)



\bibitem{HCM07}
Huang, M., Caines, P.E., Malham\'e, R.P.: Large-population
cost-coupled LQG problems with nonuniform agents: Individual-mass
behavior and decentralized $\varepsilon$-Nash equilibria. IEEE
Trans. Autom. Control \textbf{52}, 1560-1571 (2007)





\bibitem{HM16Chen}
Huang, M., Ma, Y.: Mean field stochastic games: Monotone costs
and threshold policies (in Chinese), Sci. Sin. Math. (special issue in
honour of the 80th birthday of Prof. H-F. Chen) \textbf{46}, 1445-1460 (2016)


\bibitem{HM17}Huang, M., Ma, Y.: Mean field stochastic games with binary action
spaces and monotone costs. arXiv:1701.06661v1, 2017.

\bibitem{HMcdc17}
Huang, M., Ma, Y.:
Mean field stochastic games with binary actions: Stationary threshold policies.
Proc. 56th IEEE Conference on Decision and Control, Melbourne, Australia,  pp. 27-32 (2017)


\bibitem{HMC06}
Huang, M., Malham\'e, R.P., Caines, P.E.: Large population stochastic dynamic games: Closed-loop McKean-Vlasov systems and the Nash certainty equivalence principle. Commun.  Inform.  Systems \textbf{6}, 221-251 (2006)


\bibitem{HZ18}
Huang, M., Zhou, M..: Linear quadratic mean field games: Asymptotic solvability and relation to the fixed point approach.  IEEE Transactions on Automatic Control (2018, in revision, conditionally accepted)



\bibitem{IK92}
Ito, K., Kunisch, K.: Sensitivity analysis of solutions to optimization problems in Hilbert spaces with applications
to optimal control and estimation. J. Differential Equations \textbf{99}, 1-40 (1992)


\bibitem{JR88}
Jovanovic, B., Rosenthal, R.W.: Anonymous
sequential games. Journal of Mathematical Economics \textbf{17}, 77-87 (1988)




\bibitem{JAW11}
Jiang, L., Anantharam, V., Walrand, J.:
How bad are selfish investments in network security?
 IEEE/ACM Trans. Networking \textbf{19}, 549-560 (2011)


\bibitem{K12}
Kolokoltsov, V.N.:
Nonlinear Markov games on a finite state space (mean-field and
binary interactions).
International J. Statistics Probability \textbf{1}, 77-91 (2012)




\bibitem{K89}
Kress, R.: Linear Integral Equations. Springer, Berlin (1989)




\bibitem{LL07}
Lasry, J.-M., Lions, P.-L.: Mean field games. Japan. J.
Math. \textbf{2}, 229-260 (2007)



\bibitem{LB08}
Lelarge, M., Bolot, J.: A local mean field analysis of security investments
in networks. Proc. ACM SIGCOMM NetEcon, Seattle, WA, pp. 25-30,
 2008


\bibitem{LZ08}
Li, T., Zhang, J.-F.: Asymptotically optimal decentralized control
for large population stochastic multiagent systems. IEEE Trans.
Autom. Control \textbf{53}, 1643-1660 (2008)





 \bibitem{MP10}
Manfredia, P., Posta, P.D., d\'Onofrio, A., Salinelli, E., Centrone, F., Meo, C., Poletti, P.:
Optimal vaccination choice, vaccination games, and rational exemption:
An appraisal.  Vaccine \textbf{28}, 98-109 (2010)


\bibitem{MT09}
Meyn, S., Tweedie, R. L.: Markov Chains and Stochastic Stability, 2nd ed. Cambridge University Press, Cambridge (2009)


\bibitem{MS94}
Milgrom, P., Shannon, C.:  Monotone comparative statics. Econometrica \textbf{62}, 157-80 (1994)


\bibitem{MB17}
Moon, J.,  Basar, T.: Linear quadratic risk-sensitive and robust mean field games. IEEE
Trans. Autom. Control \textbf{62},  1062-1077 (2017)


\bibitem{MS02}
M\"uller, A., Stoyan, D.: Comparison Methods for Stochastic Models and Risks.
Wiley, Chichester (2002)




\bibitem{O73}
Oniki, H.: Comparative dynamics (sensitivity analysis) in optimal control theory.
J. Econ. Theory \textbf{6}, 265-283 (1973)


\bibitem{STR18}
Saldi, N., Basar, T., Raginsky, M.:
Markov-Nash equilibria in mean-field games with discounted cost.
SIAM J. Control Optimization \textbf{56}, 4256-4287 (2018)



\bibitem{S83}
Samuelson, P.A.: Foundations of Economic Analysis, enlarged edn., Harvard University Press, Cambridge, MA (1983)

\bibitem{S73}
Schelling, T.C.: Hockey helmets, concealed weapons, and daylight saving: A study of binary
choices with externalities. The Journal of Conflict Resolution \textbf{17}, 381-428 (1973)



\bibitem{S95}
Selten, R.: An axiomatic theory of a risk dominance measure
for bipolar games with linear incentives. Games and Econ. Behav. \textbf{8}, 213-263 (1995)


\bibitem{S53}
 Shapley, L.S.: Stochastic games. Proc. Natl. Acad. Sci. \textbf{39}, 1095-1100 (1953)



\bibitem{SW93}
Sigman, K., Wolff, R.W.: A review of regenerative processes. SIAM Rev. \textbf{35}, 269-288 (1993)


\bibitem{S06}
Sun, Y.: The exact law of large numbers via Fubini extension and
characterization of insurable risks. J. Econ. Theory \textbf{126}, 31-69 (2006)

\bibitem{T98}
Topkis, D.M.: Supermodularity and Complementarity.  Princeton Univ. Press,
Princeton (1998)



\bibitem{WWA11}
 Walker, M.,  Wooders, J.,   Amir, R.:
Equilibrium play in matches: Binary Markov games. Games and Economic Behavior \textbf{71}, 487-502 (2011)

\bibitem{WBR08}
Weintraub, G.Y., Benkard, C.L., Van Roy, B.:  Markov perfect
industry dynamics with many firms. Econometrica \textbf{76}, 1375-1411 (2008)


\bibitem{Y11}
Yong, J.: Linear-quadratic optimal control problems for mean-field stochastic differential
equations. SIAM J. Control Optim.  \textbf{51},  2809-2838 (2013)



\bibitem{ZH17}
Zhou, M., Huang, M.: Mean field games with Poisson point processes and impulse control. Proc. 56th IEEE Conference on Decision and Control, Melbourne, Australia  pp. 3152-3157 (2017)



\end{thebibliography}
%

\end{document}